\documentclass[11pt]{article}
\usepackage{amsmath}
\usepackage{amsthm}
\usepackage{amssymb,amsfonts}
\usepackage[T1]{fontenc}
\usepackage[backend=biber, style=alphabetic]{biblatex}
\usepackage{hyperref}
\usepackage{latexsym}
\usepackage{mathrsfs}
\usepackage{todonotes}
\usepackage{upgreek}
\usepackage{rotating}
\usepackage{nicefrac}
\usepackage{epsfig}
\usepackage{stmaryrd}
\usepackage{setspace}
\usepackage{enumerate}
\usepackage{bbm,ifpdf,tikz}
\usepackage{mathtools}
\usepackage{tikz}
\usetikzlibrary{cd}

\renewbibmacro{in:}{}

\newtheorem{theorem}{Theorem}[section]

\newtheorem{lemma}[theorem]{Lemma}
\newtheorem{proposition}[theorem]{Proposition}
\newtheorem{corollary}[theorem]{Corollary}
\newtheorem{conjecture}[theorem]{Conjecture}

{
\theoremstyle{definition}

\newtheorem{example}[theorem]{Example}

\newtheorem{remark}[theorem]{Remark}

}

\newcommand{\LC}{\mathcal{L}^c}

\newcommand{\excise}[1]{}

\newcommand{\Id}{\operatorname{Id}}
\newcommand{\id}{\operatorname{id}}

\renewcommand{\dim}{\operatorname{dim}}

\newcommand{\D}{\mathrm{D}}

\renewcommand{\and}{\qquad\text{and}\qquad}

\newcommand{\Hom}{\operatorname{Hom}}

\newcommand{\Z}{\mathbb{Z}}
\newcommand{\Q}{\mathbb{Q}}
\newcommand{\N}{\mathbb{N}}
\newcommand{\R}{\mathbb{R}}
\newcommand{\C}{{\mathbb{C}}}

\newcommand{\cP}{\mathcal{P}}
\newcommand{\cG}{\mathcal{G}}

\newcommand{\cC}{\mathcal{C}}
\newcommand{\cD}{\mathcal{D}}
\newcommand{\cR}{\mathcal{R}}

\newcommand{\cPT}{\mathcal{PT}}
\newcommand{\cPTop}{\mathcal{PT}^\op}
\newcommand{\cRGg}{\mathcal{RG}_{\!g}}

\newcommand{\cRGgop}{\mathcal{RG}_g^\op}
\newcommand{\cRGop}{\mathcal{RG}^\op}
\newcommand{\cRGhop}{\mathcal{RG}_h^\op}

\newcommand{\cODop}{\mathcal{OD}^{\op}}

\newcommand{\cA}{\mathcal{A}}
\newcommand{\cM}{\mathcal{M}}

\newcommand{\la}{\lambda}

\newcommand{\Rep}{\operatorname{Rep}}

\newcommand{\nicktodo}{\todo[inline,color=green!20]}

\newcommand{\Conf}{\operatorname{Conf}}

\newcommand{\Aut}{\operatorname{Aut}}

\newcommand{\cGop}{\mathcal{G}^{\op}}

\newcommand{\op}{{\operatorname{op}}}

\newcommand{\cN}{\mathcal{N}}

\DeclareMathOperator{\FI}{FI}

\DeclareMathOperator{\colim}{colim}

\DeclareMathOperator{\OI}{OI}

\newcommand{\LT}{\operatorname{LT}}
\newcommand{\LQ}{\operatorname{LQ}}

\newcommand{\Tor}{\operatorname{Tor}}

\renewcommand{\a}{\mathbf{a}}

\renewcommand{\phi}{\varphi}

\begin{document}
\spacing{1.2}
\noindent{\LARGE\bf The graph minor theorem in topological combinatorics}\\

\noindent{\bf Dane Miyata}\\
 Department of Mathematics, University of Oregon,
Eugene, OR 97403\\

\noindent{\bf Eric Ramos}\\
Department of Mathematics, Bowdoin College,
Brunswick, ME 04011\\

{\small
\begin{quote}
\noindent {\em Abstract.}
We study a variety of natural constructions from topological combinatorics, including matching complexes as well as other graph complexes, from the perspective of the graph minor category of \parencite{MiProRa}. We prove that these complexes must have universally bounded torsion in their homology across all graphs of bounded genus. One may think of these results as arising from an algebraic version of the graph minor theorem of Robertson and Seymour \parencite{RSXX,RSXXIII}.
\end{quote} }

\section{Introduction}\label{sec:intro}

We study a category $ \cG_{\leq g} $ whose objects are finite connected graphs of genus at most $ g $ and
whose morphisms are built out of automorphisms, deletions, and contractions.  A precise definition of this
category will appear in Section \ref{sec:def}. We then prove finite generations results about the
representation theory of this category and apply these results in topological combinatorics. In particular we
study spaces arising from graphs that behave well with respect to automorphisms, deletions, and contractions.

Throughout this introduction, a \textbf{graph} will refer to an at most one-dimensional CW complex that is
both connected and finite. We say that a graph $G$ is a \textbf{minor} of a graph $G'$ if $G$ can be obtained
from $G'$ by a sequence of edge deletions and contractions. In their seminal series of papers, Robertson and
Seymour proved, among many other things, that the minor relation is actually a well-quasi-order
\parencite{RSXX,RSXXIII}. That is to say, in \emph{any} infinite collection of graphs, there must be a pair
where one is a minor of the other. In this paper, we prove a weakened categorical version of the graph minor
theorem which we outline below.

Fix a nonnegative integer $ g $ and write $\cG_{\leq g} $ for the category whose objects are graphs of genus at most $ g $, and whose morphisms are what are known as \textbf{minor
morphisms} (see Section \ref{sec:def}). While the precise definition of a minor morphism is a bit
technical, for the purposes of this introduction you can think of them as maps built out of edge deletions,
contractions, and graph automorphisms. In particular, one has a minor morphism $\phi:G' \rightarrow G$ if and
only if $G$ is a minor of $G'$. A $\cGop$-module over a Noetherian ring $R$ is then a covariant functor
$M:\cGop \rightarrow R\mbox{-mod}$. Concretely, a $\cGop$-module may be thought of as a collection of
$R$-modules $\{M(G)\}_{G}$, one for each graph $G$, such that whenever $G$ is a minor of $G'$, one has an
induced map $M(G) \rightarrow M(G')$. We say that a $\cGop$-module $M$ is finitely generated if there is some
finite list of graphs $\{G_i\}_{i}$ such that each $ M(G_{i}) $ is a finitely generated $ R $-module and for
any graph $G$, the $R$-module $M(G)$ is spanned by the images of the $M(G_i)$ under the aforementioned maps
induced by the minor relation. 

Our first result follows from results of \cite{PR2} and shows that $ \cGop_{\leq g} $-modules satisfy a {\bf
Noetherian Property} and should be viewed as a weakened categorical version of the graph minor theorem.

\begin{theorem}\label{thm:weakcatgraphminor}
Any submodule of a finitely generated $ \cGop_{\leq g} $-module is itself finitely generated. 
\end{theorem}

This weak categorical graph minor theorem  translates the combinatorics of well-quasi-orders present in the original graph minor theorem to an algebraic statement about submodules of finitely generated modules. While it is the case that the categorical graph minor theorem is equivalent to the original, it is also presented in a language that is more amenable to application in topology and algebra. In particular, see \parencite{PR2}\parencite{PR} for applications of the categorical graph minor theorem to Kazhdan–Lusztig polynomials of graphical matroids, as well as configuration spaces of graphs.  In the following work, we focus our attention to two topics that are important in topological combinatorics -- simplicial complexes and hyperplane arrangements -- in order to illustrate the power of the categorical graph minor theorem.

\subsection{Homology of the matching complex}

For a graph $G$, a \textbf{matching of $G$} is a collection of pairwise non-adjacent edges. The
\textbf{matching complex $\cM(G)$ of $G$} is the simplicial complex whose simplices are in bijection with
matchings on $G$. Matching complexes of complete graphs and complete bipartite graphs have been studied
extensively using a wide variety of techniques including discrete Morse theory. It is notable, however, that there is relatively little known about the matching complexes of general graphs. In this work, we aim to fill this gap in the literature by considering a new perspective on the problem: instead of focusing on $\cM (G)$ for some particular graph $G$, or even for some particular family of graphs, we consider the matching complexes of all graphs at once.

The key observation that allows us to do this is that whenever $G$ is a minor of $G'$ there is a natural way to embed the edges of $G$ into those of $G'$. Moreover, this embedding preserves the condition that the edges are disjoint, as any "undoing" of an edge deletion or contraction can only push things further apart. We therefore obtain simplicial complex maps $\cM (G) \rightarrow \cM (G')$ whenever $G$ is a minor of $G'$, which induce maps on homology $H_i(\cM (G)) \rightarrow H_i(\cM (G'))$. In particular, for any fixed $i \geq 0$, the assignment
\[
G \mapsto H_i(\cM (G))
\]
is a well-defined $\cGop_{\leq g}$-module over $\Z$. Using the categorical graph minor theorem, we will prove the following

\begin{theorem}\label{mainthm1}
The $\cGop_{\leq g}$-module
\[
G \mapsto H_i(\cM (G))
\]
is finitely generated. In particular, there exists an integer $\epsilon_{i,g}$, depending only on $i$ and $g$, such that the torsion part of $H_i(\cM (G))$ is annihilated by $\epsilon_{i,g}$, for all graphs $G$.
\end{theorem}

The statement about torsion in the above theorem should be particularly interesting, as the torsion appearing the homology of the matching complex has been a subject of considerable intrigue in recent years \parencite{ShareshWachs}\parencite{Jonsson3TorChess}\parencite{JonssonMoreTor}.

We also note that, by the very general nature of the weak categorical graph minor theorem, the conclusions of Theorem \ref{mainthm1} will remain true when the matching complex is replaced by a large variety of graph-based complexes, such as those discussed in \parencite{JonssonBook}. See Section \ref{othergraph} for more on this.

\subsection{Graphical linear subspace arrangements}

Recall that, to any graph $G$, one may define a linear subspace arrangement induced by the edge relation of its vertices. We define the complement of this arrangement in $(\C^d)^{V(G)}$ by $\Conf(G,\C^d)$. More precisely, we have
\[
\Conf(G,\C^d) = \{(x_v)_{v \in V(G)} \in (\C^d)^{V(G)} \mid x_v \neq x_w \text{ if $\{v,w\}$ is an edge of
$G$}\}.
\]

For instance, if $G = K_n$ is the complete graph, $\Conf(G,\C^d)$ is the classical configuration space of points in $\C^d$. If instead $G = K_{a,b}$ is the complete bipartite graph, then one recovers the colored configuration spaces $\widetilde{\mathcal{Z}}^D_{a+b}$ as studied by Farb, Wolfson, and Wood \parencite{FWW}. In this work we will specialize to the family of \emph{line graph complements}.

Given a graph $G$, we write $\LC(G)$ to denote the simple graph whose vertices are indexed by the edges of $G$, and whose edges indicate the corresponding edges in $G$ are non-adjacent. These graphs have been called Kneser graphs by some authors \parencite{DaMeuMi}, as the usual Kneser graph $K(n,2)$ is seen to be $\LC(K_n)$.

The same observation made above tells us that whenever $G$ is a minor of $G'$, one obtains a graph embedding
\[
\LC(G) \hookrightarrow \LC(G').
\]
This embedding induces a "forgetful" map 
\[
\Conf(\LC(G'),\C^d) \rightarrow \Conf(\LC(G),\C^d),
\]
which when composed with cohomology yields
\[
H^i(\Conf(\LC(G),\C^d)) \rightarrow H^i(\Conf(\LC(G'),\C^d)).
\]
Our second theorem therefore can be stated as follows.

\begin{theorem}\label{mainthm2}
For any fixed $i,d,g$, the assignment
\[
G \mapsto H^i(\Conf(\LC(G),\C^d))
\]
defines a finitely generated $\cGop_{\leq g}$-module over $\Z$.
\end{theorem}

One observes that the complete bipartite graph $K_{a,b}$ can (essentially) be realized as the line graph complement of some other graph. In particular, the finite generation result of Theorem \ref{mainthm2} can be seen as a generalization of some results in \parencite{FWW} (see Theorem \ref{basicallyOI}).

The paper is organized as follows. In Section \ref{sec:Grobner} we outline and expand on the theory of
Gr\"obner categories introduced by Sam and Snowden in \cite{SaSno}. In Section \ref{sec:graphs} we define the
category $ \cG_{\leq g} $ and use the results of Section \ref{sec:Grobner} to show that $ \cGop_{\leq g} $ has
a Noetherian property.  Finally in Section \ref{sec:applications} we use the Noetherian property of $
\cGop_{\leq g} $ to study spaces in topological combinatorics arising from graphs. 

\begin{remark}
In previous versions of this work, the above results were all stated without the dependency on the genus parameter $g$. These earlier versions of this paper were based on the Categorical Graph Minor Theorem of \cite{MiProRa}, and worked with the category of all graphs with minor morphisms. Unfortunately, a gap was found in the proof of the main theorem of \cite{MiProRa}, and so for now we can only state our results in the setting where genus is bounded. That being said, it is still very much the belief of the authors that the Categorical Graph Minor Theorem is true, and the techniques of this paper would then imply stronger versions of Theorems \ref{mainthm1} and \ref{mainthm2} where the parameter $g$ does not appear.
\end{remark}

\section*{Acknowledgments}
The first author was supported by NSF grants DMS-1954050, DMS-2039316, and DMS-2053243. The second author was supported by NSF grants DMS-1704811 and DMS-2137628. Both authors would like to send their thanks to Nicholas Proudfoot for his various thoughts and advice related to this project.

\section{Gr\"obner theory of categories}\label{sec:Grobner}

Let $\cC$ be an essentially small category and $k$ a ring.  We define $\Rep_k(\cC)$ to be the category of functors from $\C$
to the category of $k$-modules.
A module $\cM\in\Rep_k(\C)$ is called {\bf finitely generated} if
there exist finitely many objects $c_1,\ldots,c_r$ of $\C$ along with elements $v_i\in \cM(c_i)$ such that, for any object
$c$ of $\C$, $\cM(c)$ is spanned over $k$ by the images of the elements $v_i$ along the maps induced by
all possible morphisms $\varphi_i:c_i\to c$.  If every submodule of $\cM$ is finitely generated, then $\cM$ is said to be {\bf Noetherian}.
If every finitely generated module is Noetherian, the category $\Rep_k(\C)$ is said to be {\bf locally Noetherian}.
Sam and Snowden have developed powerful machinery for proving that module categories are locally Noetherian
which we summarize below.

Given an object $ x $ of $ \C $, let $ \C_x $ be the set of equivalence classes of morphisms out of $ x $
where $ \phi \in \Hom_{\cC}(x,y) $ is equivalent to $ \psi \in \Hom_{\cC}(x,z) $ if there exists an
isomorphism $ \rho \in \Hom_{\cC}(y,z) $ such that $ \rho \circ \phi = \psi $. The category $ \cC $ satisfies
{\bf property (G1)} for exevery object $ x $ of $ \cC $ there exists a linear order $\prec $ on $\cC_{x} $
preserved under post composition. That is to say, if $ \phi, \psi \in \Hom_{\cC}(x,y) $ and $ \phi \prec \psi
$, then for any $ \rho \in \Hom_{\cC}(y,z) $, we have $ h \circ f \prec h \circ g $. There is a natural quasi-order on $
\cC_x $ where we say $ \phi \leq \psi $ if and only if there exists a morphism $ \rho \in \Hom_{\cC}(y,z)
$ such that $ \rho \circ \phi = \psi $. Note, the orders $ \leq $ and $ \prec $ are independent. We say that $
\cC$ satisfies {\bf property (G2)} if $ \leq $ is a well quasi-order on $ \cC_x $. Namely, for any infinite
sequence $ \phi_1, \phi_2, \phi_3,\dots $ of elements of $ \cC_x $, there exists a pair of indicices $ i<j $
such that $ \phi_i \leq \phi_j $. The category $ \cC $ is {\bf Gr\"obner} if $ \cC $ is a directed category and satisfies properties (G1) and
(G2).

Let $ \cC $ and $ \cD $ be categories and $ \Phi : \cD \to \cC $ a functor. The functor $ \Phi $ satisfies {\bf property (F)} if for any object $ x $ of $ \cC $ there exist finitely many objects $ y_1, \dots, y_n $ of
$ \cD $ and morphisms $ \phi_i \in \Hom_{\cC}(x, \Phi(y_i)) $ such that for any object $ y $ of $ \cD $ and
any morphism $ \phi \in \Hom_{\cC}(x, \Phi(y)) $, there exists a morphism $ \psi \in \Hom_{\cD}(y_i, y) $ such
that $ \phi = \Phi(\psi) \circ \phi_i $. The category $ \cC $ is {\bf quasi-Gr\"obner} if there exists a
Gr\"obner category $ \cD $ and a functor $ \Phi : \cD \to \cC $ satisfying property (F).

\begin{theorem}\cite[Theorem 1.1.3]{SaSno}\label{thm:samsno}
If $ \cC $ is a quasi-Gr\"obner category and $ k $ is a Noetherian commutative ring, then $ \Rep_{k}(\cC) $ is locally Noetherian.
\end{theorem}

\subsection{Modules over algebras over categories}\label{sec:modoveralg}
% Let $\FA$ be the category of finite sets with arbitrary maps.
Let $\cC$ be a category equipped with a functor $S:\cC\to\FI$ and let $k$ be a commutative ring.
There is a natural functor from $\cC$ to $k$-algebras taking an object $c$ to the polynomial ring $$\cA_S(c)
:= k[x_e\mid e\in S(c)].$$ Equivalently, we can think of $\cA_S\in \Rep_k(\cC)$ as a module equipped with a
product $\cA_S\otimes \cA_S\to \cA_S$ that is both associative and commutative.  Let $\Rep_k(\cC, S)$ be the
category of modules over $\cA_S$.  Formally, an object of $\Rep_k(\cC, S)$ is an object $\cM\in\Rep_k(\cC)$
along with a multiplication $\cA_S\otimes \cM\to \cM$ such that the two natural maps $\cA_S\otimes
\cA_S\otimes \cM\to \cM$ coincide.  More intuitively, an object $\cM$ of $\Rep_k(\cC, S)$ consists of an
$\cA_S(c)$-module $\cM(c)$ for each object $c$ of $\cC$ and an $\cA(c)$-module map $\cM(c)\to \cM(c')$ for
each morphism $\varphi:c\to c'$, where $\cM(c')$ is an $\cA_S(c)$-module via the ring homomorphism
$\cA_S(c)\to \cA_S(c')$ induced by $\varphi$.

A module $\cM\in\Rep_k(\cC, S)$ is called {\bf finitely generated} if there exist finitely many objects
$c_1,\ldots,c_r$ of $\cC$ along with elements $v_i\in \cM(c_i)$ such that, for any object $c$ of $\cC$,
$\cM(c)$ is spanned over $\cA_S(c)$ by the images of the elements $v_i$ along the maps induced by all possible
morphisms $\varphi_i:c_i\to c$.  If every submodule of $\cM$ is finitely generated, then $\cM$ is said to be
{\bf Noetherian}.  If every finitely generated module is Noetherian, the category $\Rep_k(\cC,S)$ is said to
be {\bf locally Noetherian}. Now, we outline some basic facts about finitely generated modules and Noetherian modules, the proofs of which are completely standard.
Let $\cC$ be a category, $S:\cC\to\FI$ a functor, and $k$ a commutative ring.
For any object $c$ of $\cC$, define the {\bf principal projective}
$\cP_c\in \Rep_k(\cC,S)$ to be the module that takes an object $c'$ to the free $\cA_S(c')$-module spanned by the set $\Hom_{\cC}(c,c')$,
with maps defined via composition.

\begin{lemma}\label{presentations}
A module $\cM\in \Rep_k(\cC,S)$ is finitely generated if and only if there exists
a surjection $$\bigoplus_{i=1}^r \cP_{c_i}\twoheadrightarrow\cM$$
for some list of (not necessarily distinct) objects $c_1,\ldots,c_r$ of $\cC$.
\end{lemma}

\begin{proof}
Suppose that $c_1,\ldots,c_r$ are objects of $\cC$ and $v_i\in \cM(c_i)$ for all $i$.  These classes generate $\cM$
if and only if the map $$\bigoplus_{i=1}^r \cP_{c_i}\to\cM$$
taking $\id_{c_i}\in \cP_{c_i}(c_i)$ to $v_i$ is surjective.
\end{proof}

Recall from the introduction that a module $\cM\in\Rep_k(\cC,S)$ is {\bf Noetherian} if every submodule of $\cM$ is finitely generated.

\begin{lemma}\label{ACC}
A module $\cM$ is Noetherian if and only if every ascending chain of submodules of $\cM$ eventually stabilizes.
\end{lemma}

\begin{proof}
Suppose that $\cM$ is Noetherian and $(\cN_i\mid i\in\N)$ is an ascending
chain of submodules of $\cM$.  Let $\cN := \bigcup_{i\in \N} \cN_i\subset \cM$.
Since $\cM$ is Noetherian, $\cN$ is finitely generated.  If we choose $i$ large enough so that $\cN_i$
contains all of the finitely many generating classes, then we have $\cN_i = \cN$.

Conversely, suppose that $\cM$ has a submodule $\cN\subset \cM$ that is not finitely generated.
We will define an ascending chain of finitely generated submodules $(\cN_i\mid i\in\N)$ as follows.  Let $\cN_0 = 0$.
Once we have defined $\cN_i$, the fact that $\cN_i$ is finitely generated means that $\cN_i \subsetneq \cN$,
so we may choose an object $c_i$ of $\cC$ and an element $v_i\in \cN(c) \setminus \cN_i(c)$.
Let $\cN_i$ be the smallest submodule of $\cN$ containing both $\cN_i$ and $v_i$.  This chain of submodules
clearly does not stabilize.
\end{proof}

\begin{lemma}\label{exact}
Suppose that $0\to\cM'\to\cM\to\cM''\to 0$ is short exact sequence in $\Rep_k(\cC,S)$.  Then $\cM$ is Noetherian
if and only if both $\cM'$ and $\cM''$ are Noetherian.
\end{lemma}

\begin{proof}
If $\cM$ is Noetherian, then $\cM'$ is Noetherian by definition.  If $\cN''\subset \cM''$ is a submodule, let $\cN\subset\cM$
be the preimage of $\cN''$ in $\cM$.  Since $\cM$ is Noetherian, $\cN$ is finitely generated, thus so is $\cN''$ by Lemma \ref{presentations}.

Conversely, suppose that both $\cM'$ and $\cM''$ are Noetherian, and let $(\cN_i\mid i\in\N)$ be an ascending
chain of submodules of $\cM$.  For each $i$, let $\cN_i' := \cN_i\cap \cM'$ and let $\cN_i''$ be the image of $\cN_i$ in $\cM''$.
Since $\cM'$ and $\cM''$ are both Noetherian, Lemma \ref{ACC} tells us that there is an index $n$ such that, for all $i>n$,
$\cN'_i = \cN'_{i+1}$ and $\cN''_i = \cN''_{i+1}$.  We can then conclude that $\cN_i = \cN_{i+1}$ by applying the Five Lemma
to the following diagram:
\[\tikz[->,thick]{
\matrix[row sep=10mm,column sep=20mm,ampersand replacement=\&]{
\node (g) {$0$}; \&\node (a) {$\cN'_i$}; \& \node (b) {$\cN_i$}; \& \node (c) {$\cN''_i$}; \&\node (i) {$0$}; \\
\node (h) {$0$}; \&\node (d) {$\cN'_{i+1}$}; \& \node (e) {$\cN_{i+1}$}; \& \node (f) {$\cN''_{i+1}$}; \&\node (j) {$0$}; \\
};
\draw (g) -- (a) ;
\draw (h) -- (d) ;
\draw (a) -- (b) ;
\draw (b) -- (c) ;
\draw (a) -- (d) node[left,midway]{$=$} ;
\draw (b) -- (e) ;
\draw (c) -- (f) node[right,midway]{$=$}  ;
\draw (d) -- (e) ;
\draw (e) -- (f) ;
\draw (c) -- (i) ;
\draw (f) -- (j) ;
}\]
Thus $\cM$ satisfies the ascending chain condition, and is therefore Noetherian by Lemma \ref{ACC}.
\end{proof}

In the following sections, we build on the work of \cite{SaSno} and define what it means
for the pair  $(\cC, S) $ to be {\bf Gr\"obner} (resp. {\bf quasi-Gr\"obner}) and prove the following generalization of Theorem
\ref{thm:samsno}.

\begin{theorem}\label{thm:with sets}
Let $\cC$ be a category and $S:\cC\to\FI$ a functor.
If the pair $(\cC, S)$ is quasi-Gr\"obner, then $\Rep_k(\cC, S)$ is locally Noetherian for any Noetherian commutative ring $k$.
\end{theorem}

\begin{remark}
Theorem \ref{thm:with sets} is motivated by the work of Nagel and R\"omer \cite{NaRo}.  Though they do not make these definitions
in the same generality, they essentially prove that the pair $(\FI, \id)$ is quasi-Gr\"obner, and they use this result to show that
$\Rep_k(\FI, \id)$ is locally Noetherian for any Noetherian commutative ring $k$.  Moreover, they show that if $S_d:\FI\to\FI$
is the functor taking a set $T$ to the set of unordered $d$-tuples of distinct elements of $T$, then the pair $(\FI, S_d)$
is quasi-Gr\"obner and the category $\Rep_k(\FI, \id)$ is locally Noetherian if and only if $d\leq 1$
\cite[Proposition 4.8]{NaRo}.
\end{remark}

\begin{remark}\label{rmk:both}
Note that we have $\Rep_k(\cC) = \Rep_k(\cC, \emptyset)$,
where $\emptyset:\cC\to\FI$ is the constant functor that takes every object of $\cC$ to the empty set.
We will see that the category $\cC$ is Gr\"obner (resp. quasi-Gr\"obner) in the sense of \cite{SaSno} if and only if the pair $(\cC,\emptyset)$ is quasi-Gr\"obner.
\end{remark}

\subsection{Gr\"obner pairs}\label{sec:grobpairs}
Let $\OI$ be the category whose objects are totally ordered finite sets and whose morphisms are ordered inclusions,
and let $\Psi:\OI\to\FI$ be the functor that forgets the order on a finite set.
Let $\D$ be an essentially small category and $T:\D\to\OI$ any functor.
The purpose of this section is to define what it means for the pair $(\D,T)$ to be Gr\"obner.

A {\bf quartet} for the pair $(\D,T)$ is a quadruple $\mu = (d,d',\varphi,m)$, where $d$ and $d'$ are objects of $\D$,
$\varphi:d\to d'$ is a morphism, and $m:T(d')\to \N$ is a map of sets.
For any morphism $\psi:d'\to d''$ in $\D$, we will write
$T(\psi):T(d')\to T(d'')$ for the induced morphism in $\OI$, and we will write $$\psi(\mu) := (d,d'', \psi\circ\varphi, m_\psi),$$
where $m_\psi$ is determined by the conditions that $m_\psi\circ T(\psi) = m$ and $m_\psi$ is identically zero outside of the image of $T(\psi)$.
For any map $n:T(d')\to \N$, we will write $$\mu+n:= (d,d',\varphi, m+n).$$
If $\mu_1 = (d,d'_1,\varphi_1,m_1)$ and $\mu_2 = (d,d'_2,\varphi_2,m_2)$, we say that $\mu_1\leq \mu_2$
if there exists a morphism $\psi:d_1'\to d_2'$ and a map $n:T(d'')\to N$ such that $\mu_2 = \psi(\mu_1) + n$.

\begin{remark}
The motivation for these definitions is that, once we choose a commutative ring $k$, the quartet $\mu$ determines a monomial
$$x^m := \prod_{a\in T(d')} x_a^{m(a)} \in \cR(d')$$
along with an element
$$b_\mu := x^m\cdot \varphi \in \cP_d(d')\in\Rep_k(\D,\Psi\circ T).$$
Then $\mu_1\leq\mu_2$ if and only if $\varphi_2$ factors through $\varphi_1$ via a map $\psi$ and
we have $$b_{\mu_2} = x^n \psi(b_{\mu_1})$$
for some monomial $x^n \in \cA(d'_2)$.
\end{remark}

We say that $\mu_1$ and $\mu_2$ are {\bf equivalent} % and write $\mu_1\sim\mu_2$
if $\mu_1\leq \mu_2\leq \mu_1$.
% That is, the elements $d'_1$ and $d'_2$ should be isomorphic and this isomorphism
% should take $b_{\mu_1}$ to $b_{\mu_2}$.
For each object $d$ of $\D$, let $|\D_d^T|$ denote the poset of equivalence classes of quartets with first coordinate $d$.
Given a quartet $\mu = (d,d',\varphi,m)$, we will write $[\mu]$ to denote its equivalence class in $|\D_d^T|$.
A well-order $\prec$ of $|\D_d^T|$ is called {\bf admissible} if, given two quartets
$\mu_1 = (d,d',\varphi_1,m_1)$ and $\mu_2 = (d,d',\varphi_2,m_2)$ with the same source and target
along with a morphism $\psi:d'\to d''$ and a map $n:T(d'')\to \N$, we have
$$[\mu_1]\prec [\mu_2] \;\;\Longrightarrow \;\; [\psi(\mu_1) + n] \prec [\psi(\mu_2) + n].$$

We say that the pair $(\D,T)$ satisfies {\bf property (G1)} if, for every object $d$ of $\D$, the poset $|\D_d^T|$
admits an admissible well-order.
A poset $P$ is said to be {\bf Noetherian} if, for any sequence $(p_i\mid i\in\N)$ in $P$, there exist natural numbers $i<j$ such that $p_i\leq p_j$.
We say that the pair $(\D,T)$ satisfies {\bf property (G2)} if, for every object $d$ of $\D$, the poset $|\D_d^T|$
is Noetherian.
The category $\D$ is said to be {\bf directed} if, for any object $d$ of $\D$, the only morphism from $d$ to $d$ is the identity.
We call the pair $(\D,A)$ {\bf Gr\"obner} if $\D$ is directed and $(\D,A)$ satisfies properties (G1) and (G2).

\begin{remark}
Property (G1) for the pair $(\D,\emptyset)$ is equivalent to property (G1) for $\D$ as defined in
\cite[Section 1.1]{SaSno},
and similarly
property (G2) for the pair $(\D,\emptyset)$ is equivalent to property (G2) for $\D$.
Thus a directed category $\D$ is Gr\"obner in the sense of \cite{SaSno} if and only if the pair $(\D, \emptyset)$ is Gr\"obner.
\end{remark}

The following Proposition says that the functor $T$ does not add anything interesting to property (G1).  In other words,
the distinction between a Gr\"obner category and a Gr\"obner pair lies entirely in the property (G2).

\begin{proposition}\label{G1 is lame}
The pair $(\D,T)$ satisfies property (G1) if and only if the pair $(\D,\emptyset)$ satisfies property (G1).
\end{proposition}

\begin{proof}
If $\prec$ is an admissible order of $|\D_d^T|$, then restriction to quartets with $m=0$ gives an admissible order of $|\D_d^\emptyset|$.
Conversely, if we have an admissible order of $|\D_d^\emptyset|$, we can compare the classes of two quartets
$\mu_1 = (d,d'_1,\varphi_1,m_1)$ and $\mu_2 = (d,d'_2,\varphi_2,m_2)$ for $(\D,T)$ by first comparing the classes
of the quartets $(d,d'_1,\varphi_1,0)$ and $(d,d'_2,\varphi_2,0)$ for $(\D,\emptyset)$ and then, if they are equal, breaking the tie by
comparing $m_1$ and $m_2$ lexicographically.
\end{proof}

\subsection{Gr\"obner bases}
Let $\D$ be an essentially small category and $T:\D\to\OI$ a functor such that the pair $(\D,T)$ is Gr\"obner,
and choose an admissible well-order $\prec$ of $|\D_d^T|$ for each object $d$ of $\D$ as in the definition of property (G1).
For any pair of objects $d$ and $d'$ in $\D$, let $Q_{d,d'}$ be the set of quartets of the form
$\mu = (d,d',\varphi,m)$.  The fact that $\D$ is directed implies that the natural map from $Q_{d,d'}$ to
$|\D_d^T|$ is injective, thus $Q_{d,d'}$ is well-ordered by $\prec$.

Fix a commutative ring $k$, so that we may define the representation category $\Rep_k(\D,\Psi\circ T)$.
% For any object $d$ of $\D$, we have the principal projective module $\cP_d\in\Rep_k(\D,\Psi\circ T)$.
% For any object $d'$ of $\D$ and any map $m:T(d')\to \N$, let $$x^m := \prod_{a\in T(d')} x_a^{m(a)}\in \cR(d').$$
% For any quartet $\mu = (d,d',\varphi,m) \in Q_{d,d'}$, we define the {\bf monomial}
% $$b_\mu := x^m \cdot \varphi \in \cP_d(d').$$
For any nonzero element $$p\;\; = \sum_{\mu \in Q_{d,d'}}\la_\mu b_\mu\;\; \in\;\;\cP_d(d'),$$ we define the {\bf leading quartet}
$\LQ(p)$ to be the maximal $\mu$ with respect to the well-order $\prec$ such that the coefficient $\la_\mu\in k$ is nonzero.  If $\mu = \LQ(p)$,
we define % the {\bf leading monomial} $\LM(p) := b_\mu \in \cP_d(d')$,
the {\bf leading term} $\LT(p) := \la_\mu b_\mu \in \cP_d(d')$ and
the {\bf leading coefficient} $\LC(p) := \la_\mu \in k$.

\begin{lemma}\label{admissibility}
Suppose we have a morphism $\psi:d'\to d''$, a map $n:T(d')\to \N$, and an element $0\neq p\in \cP_d(d')$.
Then $$\LT(x^n \psi(p)) = x^n \psi(\LT(p)).$$
\end{lemma}

\begin{proof}
This is precisely the definition of admissibility of the well-order $\prec$.
\end{proof}

Given a submodule $\cN\subset\cP_d$, we define a subset $$\LQ(\cN) := \left\{[\LQ(p)]\;\big{|}\; 0\neq p\in \cN(d')\right\}\subset |\D_d^T|.$$
For each object $d'$ of $\D$, we define
$$\LT(\cN)(d') := \{0\} \cup \left\{\LT(p)\mid 0\neq p\in \cN(d')\subset \cP_d(d')\right\}.$$
For each quartet $\mu = (d,d',\varphi,m)$, we define the ideal
$$\LC(\cN,\mu) := \{0\} \cup \left\{\LC(p)\mid 0\neq p\in\cN(d')\;\text{and}\;\LQ(p) = \mu\right\}\subset k.$$
Lemma \ref{admissibility} implies that $\LT(\cN)\subset\cP_d$ is a submodule and
that we have an inclusion of ideals $\LC(\cN,\mu_2)\subset\LC(\cN,\mu_1)$ whenever $[\mu_1]\leq [\mu_2]$.

Suppose we are given a finite set $B = \left\{(d_1',p_1),\ldots,(d_r',p_r)\right\}$ of pairs with $0\neq p_i\in \cN(d_i')$ for all $i$.
We say that $B$ is a {\bf Gr\"obner basis} for $\cN$ if the module $\LT(\cN)$
is generated by the classes $\LT(p_i)$ for $1\leq i\leq r$.

\excise{
Fix a finite set $B$ as above, not necessarily a Gr\"obner basis for $\cN$.
Suppose that $0\neq p \in \cP_d(d')$ and $0\neq q\in \langle B\rangle(d')$ have the property that $\LT(p) = \LT(q)$.
In this case, we say that $p$ is {\bf reducible} and that $p-q$ is a {\bf reduction} of $p$.
If we have a finite sequence of elements $$p = q_0,q_1,\ldots,q_n = r\in\cP_d(d')$$ such that $q_{i+1}$
is a reduction of $q_i$ for all $i$ and $r$ is not reducible, then we say that $r$
is a {\bf remainder} of $p$.

\begin{lemma}\label{TFAE}
The following are equivalent:
\begin{enumerate}
\item $B$ is a Gr\"obner basis for $\cN$.
\item For every object $d'$ of $\D$ and every nonzero element $p\in\cN(d')$, $p$ is reducible.
\item For every object $d'$ of $\D$ and every element $p\in\cN(d')$, 0 is a remainder of $p$.
\end{enumerate}
\end{lemma}

\begin{proof}
The equivalence of the first two statements is tautological
\nicktodo{Not quite.}
and the fact that
statement 3 implies statement 2 is trivial, so we only need to show that statement 2 implies statement 3.
We will prove statement 3 by induction on $\LQ(p)$.  Assume for the sake of contradiction
that there exists an element $p\in\cN(d')$ that does not have 0 as a remainder, and choose such a $p$
with $\LQ(p)$ minimal.  Since $p$ is reducible, there exists $q\in \langle B\rangle(d')\subset \cN(d')$
with $\LT(p) = \LT(q)$.  Since $p$ does not have 0 as a remainder, neither does $p-q$,
which contradicts the minimality of $\LQ(p)$.
\end{proof}
}

\begin{lemma}\label{generates}
If $B$ is a Gr\"obner basis for $\cN$, then $B$ generates $\cN$.
\end{lemma}

\begin{proof}
If not, choose an element $p \in\cN(d')$ that is not in the submodule generated by $B$, and choose it in such a way that
the leading quartet $\LQ(p)$ is minimal with respect to the admissible well-order on $|\D_d^T|$.
Since $B$ is a Gr\"obner basis, we may choose an index $i$, a morphism $\psi:d_i'\to d'$,
a function $n:T(d')\to \N$, and a scalar $\la\in k$ such that $\LT(p) = \la x^n \psi(\LT(p_i))$.
By Lemma \ref{admissibility}, this is equal to $\LT(\la x^n \psi(p_i))$.
But then $p - \la x^n \psi(p_i)$ is not in the submodule generated by $B$ and has a leading quartet strictly smaller than $\LQ(p)$,
which gives a contradiction.
\end{proof}

\excise{
For any submodule $\cN\subset\cP_d$ and quartet $\mu = (d,d',\varphi,m)$, let
$$\LC(\cN,\mu) := \{0\} \cup \left\{\LC(p)\mid 0\neq p\in\cN(d')\right\} \subset k.$$
Note that $\LC(\cN,\mu)$ is an ideal in $k$, and if $\mu_1 \leq \mu_2$, then $\LC(\cN,\mu_1)\subset\LC(\cN,\mu_2)$.
}

\begin{proposition}\label{noetherian projectives}
Suppose that the ring $k$ is Noetherian.
For every object $d$ of $\D$, the principal projective $\cP_d\in\Rep_k(\D,\Psi\circ T)$ is Noetherian.
\end{proposition}

\begin{proof}
By Lemma \ref{generates}, it is sufficient to show that every submodule $\cN\subset\cP_d$ has a Gr\"obner basis.
By property (G2), the set $\LQ(\cN)\subset |\D_d^T|$ has only finitely many minimal elements with respect to the partial
order.  Choose finitely many quartets $\mu_1,\ldots,\mu_r$
representing these minimal classes, and write
$\mu_i = (d,d_i',\varphi_i, m_i)$.
For each $i$, the fact that $k$ is Noetherian implies that the ideal
$\LC(\cN,\mu_i)$ is generated by finitely many elements $$\la_i^1,\ldots,\la_i^{s_i}\in k.$$
For each $j\leq s_i$, choose an element $0\neq p_i^j\in \cN(d_i')$ with $\LT(p_i^j) = \la_i^j b_{\mu_i}$,
and let $$B := \{(d_i', p_i^j)\mid 1\leq i\leq r,\; 1\leq j\leq s_i\}.$$
We claim that $B$ is a Gr\"obner bases for $\cN$.

Let $0\neq p\in\cN(d')$ be given; we will show that $\LT(p)$ is in the submodule of $\cP_d$ generated by the classes $\LT(p_i^j)$.
Let $\nu := \LQ(p)$.  By definition of the quartets $\mu_1,\ldots,\mu_r$, there exists an index $i$ such that $[\mu_i]\leq [\nu]$.
That means that we can choose a morphism $\psi:d_i'\to d'$ and a map $n:T(d')\to \N$ such
that $\nu = \psi(\mu_i) + n$.  Since $[\mu_i]\leq [\nu]$, we have $\LC(\cN, \nu)\subset \LC(\cN, \mu_i)$, and therefore
there exist elements
$\zeta_i^1,\ldots,\zeta_i^{s_i}\in k$ such that $$\LC(p) = \zeta_i^1\la_i^1 + \cdots + \zeta_i^{s_i}\la_i^{s_i}.$$
% Let $$q :=  x^n \psi(\zeta_i^1 p_i^1 + \cdots + \zeta_i^{s_i} p_i^{s_i}).$$
Then
\begin{eqnarray*}
\LT(p) &=& \LC(p) b_\nu\\
&=& (\zeta_i^1\la_i^1 + \cdots + \zeta_i^{s_i}\la_i^{s_i}) b_{\psi(\mu_i) + n}\\
&=& x^n \psi(\zeta_i^1 \la_i^j b_{\mu_i} + \cdots + \zeta_i^{s_i} \la_i^j b_{\mu_i})\\
&=& x^n \psi\left(\zeta_i^1 \LT(p_i^1) + \cdots + \zeta_i^{s_i} \LT(p_i^{s_i})\right)
\end{eqnarray*}
is in the submodule of $\cP_d$ generated by the classes $\LT(p_i^j)$.
\end{proof}

\begin{corollary}\label{thm:grob}
Let $\D$ be an essentially small category, $T:\D\to\OI$ a functor, and $k$ a Noetherian commutative ring.
If the pair $(\D, T)$ is Gr\"obner, then $\Rep_k(\D, \Psi\circ T)$ is locally Noetherian.
\end{corollary}

\begin{proof}
Suppose that $\cM\in \Rep_k(\D, \Psi\circ T)$ if finitely generated.  By Lemma \ref{presentations}, $\cM$ is a quotient
of a direct sum of principal projectives.  Proposition \ref{noetherian projectives} tells us that each of these principal
projectives is Noetherian, and Lemma \ref{exact} then tells us that the same is true of $\cM$.
\end{proof}

\subsection{Quasi-Gr\"obner pairs}
Let $\Phi:\D\to\cC$ be a functor.
Following Sam and Snowden \cite[Definition 3.2.1]{SaSno}, we say that $\Phi$ has {\bf property (F)}
if, for any object $c$ of $\cC$, there exist finitely many objects $d_1,\ldots,d_r$ of $\D$ along with morphisms $\varphi_i:c\to \Phi(d_i)$
such that, for any object $d$ of $\D$ and any morphism $\psi:c\to \Phi(d)$, there exists an $i$ and a morphism $\rho:d_i\to d$
such that $\psi = \Phi(\rho)\circ  \varphi_i$.
% We say that $\cC$ is {\bf quasi-Gr\"obner} if there exists a Gr\"obner category $\D$
% and an essentially surjective functor $\Phi:\D\to\cC$ with property (F).
Given a functor $S:\cC\to\FI$, we say that the pair $(\cC,S)$
is {\bf quasi-Gr\"obner} if there exists a Gr\"obner pair $(\D,T)$ and an essentially surjective functor
$\Phi:\D\to\cC$ with property (F) such
that $S\circ \Phi$ is naturally isomorphic to $\Psi\circ T$.

\begin{remark}
The pair $(\cC, \emptyset)$ is quasi-Gr\"obner if and only if the category $\cC$ is quasi-Gr\"obner in the sense
of \cite{SaSno}.
\end{remark}

Let $\Phi:\D\to\cC$ and $S:\cC\to\FI$ be any functors.  For any commutative ring $k$, we have an exact functor
$\Phi^*:\Rep_k(\cC,S)\to \Rep_k(\D,S\circ\Phi)$ that takes a module $\cM\in\Rep_k(\cC,S)$ to
$$\Phi^*\cM := \cM\circ\Phi \in\Rep_k(\D,S\circ\Phi).$$

\begin{proposition}\label{prop:pullback}
Let $\Phi:\D\to\cC$ be a functor with property $F$, let $S:\cC\to\FI$ be any functor, and let $k$ be a commutative ring.
If $\cM\in\Rep_k(\cC,S)$ is finitely generated, then $\Phi^*\cM\in\Rep_k(\D,S\circ\Phi)$ is finitely generated.
\end{proposition}

\begin{proof}
Since $\cM$ is finitely generated, Lemma \ref{presentations} tells us that $\cM$ is a quotient of a direct sum of principal projectives.
Since $\Phi^*$ is exact, $\Phi^*\cM$ is a quotient of a direct sum of pullbacks of principal projectives.
Thus, it is sufficient to show that, for any object $c$ of $\cC$, $\Phi^*\cP_c$ is finitely generated.
Choose finitely many objects $d_1,\ldots,d_r$ of $\D$ along with morphisms $\varphi_i:c\to \Phi(d_i)$
as in the definition of property (F).  Consider the maps
$$\cP_{d_i} \to \Phi^*\cP_{\Phi(d_i)}\to\Phi^*\cP_c,$$
where the first map is induced by $\Phi$ and the second is induced by $\varphi_i$.  Property (F) says precisely
that the direct sum map $$\bigoplus_{i=1}^r \cP_{d_i}\to \Phi^*\cP_c$$
is surjective, which implies that $\Phi^*\cP_c$ is finitely generated.
\end{proof}

\begin{proof}[Proof of Theorem \ref{thm:with sets}.]
Let $(\cC,S)$ be quasi-Gr\"obner pair.
That means that there exists a Gr\"obner pair $(\D,T)$, an essentially surjective
functor $\Phi:\D\to\cC$ with property (F), and a natural isomorphism $\Psi\circ T \cong S\circ\Phi$.
Fix a commutative ring $k$, a finitely generated module $\cM\in\Rep_k(\cC,S)$, and a submodule $\cN\subset\cM$.
We need to prove that $\cN$ is finitely generated, as well.

Proposition \ref{prop:pullback} tells us that $\Phi^*\cM\in \Rep_k(\D,S\circ\Phi)\simeq \Rep_k(\D,\Psi\circ T)$
is finitely generated, and Corollary \ref{thm:grob} then implies that $\Phi^*\cN\subset\Phi^*\cM$ is also
finitely generated.
Choose a generating set consisting of objects $d_1,\ldots,d_r$ of $\D$ and elements $v_i\in \Phi^*\cN(d_i)$.
This means that, for any object $d$ of $\D$, $\Phi^*\cN(d)$ is spanned over $\cA(d)$ by the images of the elements $v_i$
along the maps induced by all possible morphisms $\varphi_i:d\to d_i$.
This is equivalent to saying the $\cN(\Phi(d))$ is spanned over $\cA(\Phi(d))$ by the images of the elements $v_i\in\cN(\Phi(d_i))$
along the maps induced by all morphisms $\Phi(\varphi_i):\Phi(d)\to\Phi(d_i)$.  Since $\Phi$ is essentially surjective, this means that
$\cN$ is finitely generated.
\end{proof}

\section{Graphs}\label{sec:graphs}
We define the category $ \cG_{\leq g} $ and use the results of Section \ref{sec:Grobner} to prove
Noetherianity results about the representation theory of this category.  

\subsection{Defining the graph categories}\label{sec:def}
A {\bf directed graph} is a quadruple $(V,A,h,t)$, where $V$ and $A$ are finite sets (vertices and arrows),
and $h$ and $t$ are each maps from $A$ to $V$ (head and tail).  A {\bf graph} is a quintuple $(V,A,h,t,\sigma)$,
where $(V,A,h,t)$ is a directed graph and $\sigma$ is a fixed-point-free involution of $A$ with the property that $h = t\circ\sigma$.
%  (Note that, since $\sigma$ is an involution, these two equations are equivalent.)
% If $(V,E,h,t)$ is a directed graph, elements of $E$ are called {\bf edges}.
If $(V,A,h,t,\sigma)$ is a graph, elements of the quotient $A/\sigma$ are called {\bf edges}.
Given a directed graph $D = (V,A,h,t)$, we define the {\bf underlying graph} $\bar D = (V,\bar A,h,t,\sigma)$, where
$\bar A = A \times\{\pm 1\}$, $h(a,1) = h(a) = t(a,-1)$, $t(a,1) = t(a) = h(a,-1)$, and $\sigma$ acts by toggling the second factor.
We will usually suppress $h$ and $t$ from the notation and simply write $(V,A,\sigma)$ for a graph.

\begin{remark}
This might seem to be an unnecessarily complicated definition of a graph.  For example, one might try defining
a graph to consist of a vertex set, and edge set, and a map from edges to unordered pairs of vertices.  However,
we want a graph with a loop to have a nontrivial automorphism
that reverses the orientation of the loop.  It is difficult to formalize this with the unordered pair definition.
\end{remark}

If $G = (V,A,\sigma)$ is a graph and $v,v'\in V$, a {\bf walk} in $G$ from $v$ to $v'$ is a finite sequence $(a_1,\ldots,a_n)$
of arrows with $t(a_1)=v$, $h(a_n)=v'$, and $h(a_i)=t(a_{i+1})$ for all $1\leq i< n$.  A {\bf path} in $G$ from $v$ to $v'$ is a walk from $v$ to $v'$
of minimal length.  We say that $G$ is {\bf connected}
if there exists at least one path between any pair of vertices, and we say that $G$ is a {\bf forest} if there exists at most one
path between any pair of vertices.  A nonempty connected forest is called a {\bf tree}.

Let $G = (V,A, \sigma)$ and $G' = (V',A', \sigma')$ be graphs.  A {\bf minor morphism} from $G$ to $G'$ is
a map $$\varphi:V\sqcup A\sqcup\{\star\} \to V'\sqcup A'\sqcup\{\star\}$$ satisfying the following properties:
\begin{itemize}
\item $\varphi(\star) = \star$.
\item For every vertex $v\in V$, $\varphi(v) \in V'$.
\item For every arrow $a \in A$, $\varphi\circ \sigma(a) = \sigma' \circ \varphi(a)$ where $ \sigma' $
	acts trivially on $ V' \sqcup \{\star\} $. 
\item For every arrow $a'\in A'$, there exists a unique arrow $a\in A$ with $\varphi(a) = a'$.
\item If $\varphi(a) \in A$, then $\varphi\circ h(a) = h'\circ\varphi(a)$ and $\varphi\circ t(a) =
	t'\circ\varphi(a)$.
\item If $\varphi(a) \in V$, then $\varphi\circ h(a) = \varphi(a) = \varphi\circ t(a)$.
\item For every $v'\in V'$, $\varphi^{-1}(v')$ consists of the edges and vertices of a tree.
\end{itemize}
The edges that map to vertices are called {\bf contracted edges} and the edges that map to $\star$ are
called {\bf deleted edges}. Note that the
second condition and third conditions imply that the edges of $G$ that are neither deleted nor contracted map bijectively
to the edges of $G'$.  In particular,
$\varphi$ induces an injection $\varphi^*:A'/\sigma' \to A/\sigma$.

\begin{remark}\label{restriction to edges}
If $G$ and $G'$ are connected graphs, then a minor morphism $\varphi:G\to G'$ is determined by its restriction
to the set of arrows of $G$.  (Connectedness is necessary because the graph with two vertices and no arrows
has a nontrivial automorphism swapping the vertices.) However, it is convenient to define $\varphi$ on the
whole set $V\sqcup A\sqcup\{\star\}$ so that minor morphisms can be composed simply by composing functions.
Note that $\varphi$ is {\em not} determined by the map on edges $\varphi^*$.  To see this, consider the
example where $G$ has two vertices and two parallel edges between them, and $G'$ consists of a single vertex
with no edges.  There are two minor morphisms from $G$ to $G'$, corresponding to the choice of which edge is
deleted and which edge is contracted.
\end{remark}

Let $\cG$ denote the category whose objects are nonempty connected graphs and whose morphisms are minor morphisms.

\begin{conjecture}\label{conj:catgraphminor}
The category $ \cGop $ is quasi-Gr\"obner.
\end{conjecture}

\begin{remark}\label{rmk:conjecture}
The Gr\"obner cover of $ \cGop $ should be the category $ \cODop $ of directed graphs whose arrows are ordered
with opposite minor morphisms that preserve the order of arrows. That is if $ \phi: (V,A,h,t) \to
(V',A',h',t') $ is a minor morphism of directed graphs whose arrows are ordered, then we also require that
whenever $ a_1' \leq a_2' $ in $ A' $, then $ \phi^*(a_1') \leq \phi^*(a_2') $ in $ A $. However, to prove
that $ \cODop $ satisfies property (G2) one would need a stronger version of Robertson and Seymour's labeled
graph minor theorem \cite[1.7]{RSXXIII} where the order of the labels on edges is preserved in the above
sense.  If Conjecture \ref{conj:catgraphminor} is proven, then all of the results in Section
\ref{sec:applications} can be upgraded from statements about graphs with genus at most $ g $ to statements
about all graphs with no restriction on genus. 
\end{remark} 

In light of Remark \ref{rmk:conjecture}, we restrict our attention to a family of full subcategories of $
\cGop $. The {\bf combinatorial genus} of a graph $ G = (V,A,\sigma) $ is $ |A/\sigma| -|V| + 1 $. For any $ g
\geq 0$, let $ \cG_{\leq g} $ be the full subcategory of $ \cG $ whose objects are non-empty
connected graphs of
genus at most $ g $. 

\begin{theorem}\label{thm:weakgraphminor}
For any positive integer $ g $, $ \cGop_{\leq g} $ is quasi-Gr\"obner. 
\end{theorem}

The proof of Theorem \ref{thm:weakgraphminor} relies on the work in \cite{PR2} which we summarize
here. 

A {\bf rooted tree} is a pair $ (T,v) $ where $ T $ is a tree and $ v $ is a vertex of $ T $ called the
{\bf root}. There is a natural partial order on the vertex set of rooted tree where $ w \leq u $ if and only
if the unique path from $ w $ to the root passes through $ u $. A {\bf direct descendant} of a vertex $ w $ is
a vertex covered by $ w $ with respect to this partial order. A {\bf planar rooted tree} is a rooted tree
equipt with a linear order on the set of direct descendants of each vertex. Note that this gives a depth
first linear order on the set of vertices.   

Let  $ L $ be a finite set.  An
{\bf $L$-labeled planar rooted tree} is a triple $ (T,v, \ell) $ where $ (T,v) $ is a rooted tree and $
\ell $ is a function from the set of vertices of $ T $ to $ L $. 

In a nonempty connected graph $ G $, a {\bf spanning tree} for $ G $ is a subgraph that is a tree and contains
all vertices of $ G $. Now, for each genus $ g $, fix once and for all a graph $ R_g $ with one vertex and $ g
$ loops. Define a {\bf rigidified graph} of genus $ g $ to be a quadruple $ (G,T,v, \tau) $ where  $ G $ is a
graph of genus $ g $, $ (T,v) $ is a planar rooted spanning tree of $ G $, and $ \tau : G \to R_g$ is a minor
morphism where the contracted edges are exactly the edges of $ T $. Note that rigidified graphs come equipt
with a linear order on the vertices coming from their planar rooted spanning trees. A {\bf morphism of
rigidified graphs} $ (G, T, v, \tau) \to (G', T', v', \tau ') $ between rigidified graphs of genus $ g $ is a
minor morphism $ \phi : G \to G' $ that restricts to a minor morphism $ T \to T' $ such that $ \phi(v) = v' $,
and if $ w' \leq u' $ under linear order on vertices of $ G' $, then the smallest vertex in the preimage of $
w' $ comes before the smallest vertex in the preimage on $ u' $ under the linear order on the vertices of $ G
$. Note that a minor morphism between graphs of the same genus necessarily has no deleted edges.

Let $ \cRGg $ be the category whose objects are rigidified graphs of genus $ g $ and whose morphisms are minor
morphisms of rigidified graphs of genus $ g $.
Let $ \cRGop_{\leq g} $ denote the category whose objects are rigidified graphs of genus at most $ g $ and
whose morphisms are minor morphisms of rigidified graphs between graphs of the same genus. One can think of $
\cRGop_{\leq g}$ as a kind of disjoint union of the finitely many categories $ \cRGhop $ for $ h \leq g $.

\begin{theorem}[\cite{PR2}] 
For any $ g \geq 0 $, the category $ \cRGgop $ is Gr\"obner.
\end{theorem}

\begin{corollary}\label{cor:leqgrob}
For any $ g \geq 0 $, the category $ \cRGop_{\leq g} $ is Gr\"obner. 
\end{corollary}
\begin{proof}
It is clear from the definitions that $ \cRGop_{\leq g} $ satisfies properties (G1) and (G2) since $
\cRGhop $ satisifies both properties for each of the finitely many $ h \leq g $.
\end{proof}

Note that there is a forgetful functor \[
\Phi_{\leq g} : \cRGop_{\leq g} \to \cGop_{\leq g} \]  
which sends a rigidified graph to its underlying graph and a minor morphism of rigidified graphs to the underlying
minor morphism of graphs.

\begin{lemma}\label{lem:propertyF}
The functor $ \Phi_{\leq g} $ satisfies property (F).
\end{lemma}

\begin{proof} This proof is similar to the proof of \cite[Lemma
	3.11]{PR2}. The fact that $ \Phi_{\leq g} $ is essentially surjective is clear as every graph can be given a rigidified graph structure.

Let $ G $ be a graph of genus $ h \leq g $ with exactly $ n $ edges. To show $ \Phi_{\leq g} $ has property (F), we
need to produce a finite collection $ (G_i, T_i, v_i , \tau_i) $ of Rigidified graphs of genus at most $ g $
along with minor morphisms  $\phi_{i} : G_i \to G $ such that, given any rigidified graph $ (G',
T', v', \tau')$ and minor morphism $ \phi: G' \to G $, there exists an index $ i $ and a morphism of
rigidified graphs $ \rho: (G', T', v', \tau') \to (G_i, T_i, v_i, \tau_i)$ with $ \phi = \phi_{i} \circ
\Phi_{\leq g}(\rho) $. 

Consider all possible rigidified graphs with at most $ n + 2g - h $ edges and genus at least $ h $ and at most
$ g $. For each isomorphism class of such rigidified graph, choose a representative $ (G'', T'', v'', \tau'')
$ and add it to our collection $ (G_i, T_i, v_i, \tau_i) $ as many times as there are minor morphisms $ G''
\to G $. Then, take our collection of minor morphisms $ \phi_i $ so that for each
$ i $, every
possible minor morphism $ G_i \to G $ appears in our collection. Note, there are only finitely many isomorphism classes of rigidified graphs with at most $ n+2k-h $
in $ \cRGhop $ and for such a rigidified graph there are only finitely many minor morphisms from the
underlying graph to $ G $. Thus our collection of rigidified graphs and minor morphisms is indeed finite. 

Now, fix a rigidified graph $ (G', T', v', \tau') $ with genus $ k $ where $ h \leq k \leq g $ and a minor
morphism $ \phi: G' \to G $. Let $ E' $ be the set of edges that are contracted under $ \phi $. In a morphism
of rigidified graphs, we are only allowed to contract edges in the designated spanning tree. Thus, let $ \rho
$ be the morphism of rigidified graphs  $ \rho:(G' , T', v', \tau') \to
(G' / (E' \cap T'), T' / (E' \cap T'), v', \tau')$ corresponding to contracting the edges of $ E' \cap T' $.
Clearly $ \phi $ factors through $ \Phi_{\leq g}(\rho) $. All that we must verify is that $ (G'/(E' \cap T'), T' /
(E' \cap T'), v', \tau') $ is a member of our collection. In otherwords, we must show that $ G' /(E' \cap T')
$ has at most $ n + 2g - h $ edges. Let $ n' $ be the number of edges of $ G' $. Since deleting an edge lowers
the genus of the graph by 1, we know that under $ \phi $, exactly $ k-h $ edges are deleted. Thus, $ n' = n +
k-h+|E'|$ or equivalently, \[
|E'| = n' - (n + k - h).
\] Furthermore, we know that the number of edges in $ T' $ is equal to $ n' - k $. Thus, \[
|E' \cap T'| \geq n' - (n + k - h) - k = n' - n -2k + h. 
\]
This implies that the number of edges in $ G' / (E' \cap T') $ is at most \[ 
n' - (n' - n - 2k + h) = n + 2k + h
\] which is no more than $ n + 2g + h $ since $ k \leq g $. 
\end{proof}

\subsection{Edge functors}

Recall that given a minor morphism of graphs $ \phi: (V,A,\sigma) \to (V',A',\sigma') $, we get an inclusion
of edge sets $ \phi^* : A'/\sigma' \to A/\sigma $. Thus, we have an {\bf edge functor} 
\[E: \cGop_{\leq g} \to \FI \]
taking a graph $ G $ to its set of edges. Similarly, the edges of a rigidified graph are totally ordered and
given minor morphism of rigidified graphs, the pullback inclusion on edge sets preserves this order. This
gives an ordered edge functor 
\[\Omega : \cRGop_{\leq g} \to \OI \] 
taking a rigidified graph to its ordered set of
edges. Recall from Section \ref{sec:grobpairs} that $ \Psi: \OI \to \FI $ is the forgetful functor that sends
an ordered set to its underlying (unordered) set.

\begin{lemma}\label{lem:natiso}
$ \Psi \circ \Omega  $ and $ E \circ \Phi_{\leq g} $ are naturally isomorphic functors $ \cRGop_{\leq g} \to \FI
$.
\end{lemma}
\begin{proof}
This is clear from the definitions of the functors.
\end{proof}

\begin{theorem}\label{thm:grobpair}
The pair $ (\cRGop_{\leq g} , \Omega) $ is Gr\"obner. 
\end{theorem}

To prove Theorem \ref{thm:grobpair} we will need a labeled version of a planar rooted tree. This is a slightly
more general version of the notion of an $ S $-labeled planar rooted tree in \cite{PR2}. Let $ L $ be a set equipt with a well quasi order $ \leq $. An $ L $-labeled planar rooted tree is a triple $ (T,v,\ell) $ where $
(T,v) $ is a planar rooted tree and $ \ell $ is a function from the set of vertices of $ T $ to $ L $. An
{\bf $L$-labeled minor morphism of planar rooted trees} or {\bf $L $-labeled minor morphism} $ (T,v,\ell) \to (T', v', \ell') $ is a minor morphism $
\phi: (T,v) \to (T',v') $ of rigidified graphs of genus 0 (note planar rooted trees are exactly the rigidified
graphs of genus 0) such that $ \ell'(w') \leq \ell'(\max \phi^{-1}(w')) $ where $ \max \phi^{-1}(w') $ is the
first vertex in the preimage of $ w' $ under $ \phi $ with respect to the natural depth first order on the
vertices of $T$. Let $ \cPT_{L} $ denote the category whose objects are $ L $-labeled planar rooted trees and
whose morphisms are $L$-labeled minor morphisms. For a fixed $ L $-labeled planar
rooted tree $ (T,v,\ell) $, we may give a quasi oreder $ \leq $ to the $L$-labeled minor morphisms into $ (T,v,\ell) $.
Namely, if $ \phi': (T',v',\ell') \to (T,v,\ell) $ and $ \phi'': (T'', v'', \ell'') \to (T,v,\ell) $ are
$L$-labeled minor morphisms, then $ \phi \leq \phi'' $ if and only if  there exists an $ L $-labeled minor
morphism $ \phi: (T'',v'',\ell'') \to (T',v',\ell') $ such that $ \phi'' = \phi' \circ \psi $. Let $
|(\cPTop_{L})_{(T,v,\ell)}| $ denote the poset of equivalence classes of $ L $-labeled minor morphisms into
$ (T,v,\ell) $ under this quasi order.

\begin{lemma}\label{lem:labeledkruskal}
	Let $ L $ be a well quasi ordered set. We partially order the isomorphism classes of $ L $-labeled planar rooted
	trees where $ [(T', v' , \ell')] \leq [(T, v, \ell)] $ if and only if there is a minor morphism of $ L
	$-labeled planar rooted trees $ (T, v, \ell) \to (T', v', \ell')$. 
\end{lemma}

\begin{corollary}\label{cor:labeledkruskal}
	Let $ L $ be a well quasi ordered set. Fix an $ L $-labeled planar rooted tree $ (T,v,\ell) $. The poset $
	|(\cPTop_{L})_{(T,v,\ell)}| $ is a well partial order.
\end{corollary}

\begin{remark}
Note that the notion of $ L $-labeled planar rooted tree and $ L
$-labeled minor morphism are slight generalizations of the notions of an $ S $-labeled planar rooted tree and
an $ S $-labeled contraction defined in \cite{PR2} where the set of labels was an (unordered) finite set.
Nevertheless, Lemma \ref{lem:labeledkruskal} and Corollary \ref{cor:labeledkruskal} are analogues of and have nearly identical proofs to \cite[Theorem 3.6]{PR2} and \cite[Corollary 3.7]{PR2}
respectively. Thus, we omit their proofs here.
\end{remark}

Recall from Section \ref{sec:grobpairs} that a quartet $\mu = (R,R',\phi, m) $ for the pair $ (\bigsqcup_{h\leq g} \cRGhop , OE) $ consists of
two rigidified graphs of genus at most $ g $, $R = (G,T,v,\tau) $, $R'= (G',T',v',\tau') $, a minor morphism
of rigidified graphs $ \phi : R' \to R $, and a map of sets $ m $ from the edges of $ R' $ to $ \N $. We
should think of $ m $ as assigning to each edge of $ R' $ a natural number. For any
fixed rigidified graph $ R$, We also have a quasi order on quartets whose first coordinate is $ R $ where $
(R,R', \phi' , m') \leq (R, R'', \phi'', m'') $ if there exists a minor morphism of rigidified graphs $ \psi :
R'' \to R' $ such that $ \phi'' = \phi' \circ \psi $ and if $ e' $ is an edge in $ R' $, $ m'(e') \leq
m''(\phi^*(e')) $ where $ \phi^* $ is the natural inclusion of the edges of $ R' $ to the edges of $ R'' $. 
Denote the corresponding poset of equivalence classes under this quasi order by $ |(\cRGop_{\leq g})^{OE}_{R}|
$.

Fix a rigidified graph $ R = (G, T,v, \tau) $ with genus $ h \leq g $. Now for a suitable choice of $ L $, we wish to encode each quartet with first coordinate $ R $ as an $ L
$-labeled planar rooted tree so that the poset $ |(\cRGop_{\leq g} )^{OE}_{R} | $ is equivalent to the poset $
| (\cPTop_{L})_{(T,v,\ell)}| $ where $ (T,v) $ is the planar rooted tree of $ R $ and $ \ell $ is some suitable
label. To this end, let $ L = (\N \cup \{\star\})^{2h} \times \N $. The order on $ (\N \cup \star) $ comes
from the usual order order on $ \N $ along with setting $ \star $ to be incomparable to all elements of $ \N $. Then,
the order on $ L $ is the usual order on the cartesian product of posets.

Given a quartet of the form $ (R, R', \phi, m) $, where $ R' = (G', T', v', \tau') $ the corresponding $ L
$-labeled planar rooted tree will be of the form $ (T',v', \ell') $ for some labeling $ \ell' $. Note that $ R' $ must also have genus $
h$ since we only have morphisms of rigidified graphs between rigidified graphs of the same genus. Thus, $ R' $
has $ h $ extra edges not in $ T' $. The planar rooted structure of $ (T', v') $ gives an orientation and
ordering to these $ h $
extra edges (we orient them from smaller to larger vertex and then order them by the order on their
terminating vertex). Call these extra edges $ e'_1, \dots, e'_h $.  For each $ 1 \leq i \leq h $, let $
w'_{2i-1} $ be the vertex at which $ e_i $ originates and let $ w'_{2i} $ be the vertex at which $ e_i $
terminates. Then for each vertex $ w'$ of $ T' $, and each $ 1 \leq j \leq 2h $, define the $ j $th coordiante
of $ \ell'(w') $ to be $ m(e'_j) $ if $ w \geq w_j $ and $ \star $ otherwise. Then for each edge $ e' $ in $ T'
$, if $ w' $ is the vertex of $ e' $ further from the root, set the last coordinate of $ w' $ to $ m(e')
 $. Finally set the last coordinate of $ \ell'(v') $ to 0. The intuition behind this labeling is that the
 first $ 2h $ coordinates encode the location and weights given by $ m $ for the $ h $ edges not in $ T' $.
 The last coordinate encodes the weights given by $ m $ of the edges in $ T' $. 
 Finally, let $ \ell $ be the fixed labeling of $ (T,v) $ corresponding to the quartet $ (R,R, \Id_{R}
 , 0) $.

 \begin{lemma}\label{lem:labels1}
 Let $ (R, R', \phi, m) $ be a quartet and $ (T', v',\ell' ) $ be the corresponding $ L $-labeled planar
 rooted tree as defined above. Then $ \phi $ induces an $ L $-labeled morphism $ \phi_L : (T',v', \ell') \to
 (T,v,\ell) $.
 \end{lemma}
 \begin{proof}
 One simply uses a nearly identical argument to the one given in the proof of Lemma \cite[Lemma 3.8]{PR2}.
 \end{proof}

\begin{lemma}\label{lem:labels2}
Let $ \mu' = (R, R', \phi', m') $ and $\mu'' =  (R, R'',\phi'', m'') $ be quartets with corresponding $ L
$-labeled planar rooted trees $ (T',v',\ell') $ and $ (T'', v'', \ell'') $ and define $ \phi'_L $ and $ \phi''_L
$ as in Lemma \ref{lem:labels1}. Then $ \mu' \leq \mu''$ if and
only if $\phi'_L \leq \phi_L''$ in $|(\cPTop_{L})_{(T,v,\ell)}|$
\end{lemma}
\begin{proof}
Note that $ \mu' \leq \mu'' $ means there exists a minor morphism of rigidified graphs $ \psi: R'' \to R' $
such that $ \phi'' = \phi' \circ \psi $ and for any edge $ e' $ of $ R' $, $ m'(e') \leq m''(\psi^*(e')
) $. 
Using a nearly identical argument to that given in the proof of \cite[Lemma 3.8]{PR2} we see that such a minor morphism $ \psi $
is equivalent to an $L$-labeled minor morphism of $ L $-labeled planar rooted trees $ \psi_L : (T'', v'', \ell'') \to
(T',v',\ell') $ such that $ \phi''_L = \phi'_L \circ\psi_L $.  This is equivalent to saying $ \phi'_L \leq
\phi''_L $.  
\end{proof}

\begin{proof}[proof of Theorem \ref{thm:grobpair}]
That the pair $ (\cRGop_{\leq g} , \Omega) $ satisfies property (G1) follows from Corollary \ref{cor:leqgrob} and
Proposition \ref{G1 is lame}. The fact that $ (\cRGop_{\leq g} , OE) $ satisifies property (G2) follows from
Lemma \ref{lem:labels2} and Corollary \ref{cor:labeledkruskal}.
\end{proof}

\begin{theorem}\label{thm:weakgraphminorpairs}
The pair $ (\cGop_{\leq g} , E) $ is quasi Gr\"obner
\end{theorem}
\begin{proof}
This follows from Lemmas \ref{lem:propertyF} and \ref{lem:natiso}.
\end{proof}

\subsection{Finite generation}\label{fgcor}

Theorems \ref{thm:weakgraphminor} and \ref{thm:weakgraphminorpairs} tell us that the categories of
representations $\Rep_{k}(\cGop_{\leq g}) $ and $ \Rep_{k} (\cGop_{\leq g} , E) $ are locally Noetherian. That
is, every submodule of any finitely generated module of either of these categories, is itself finitely
generated.

The properties of finitely generated $\cGop$-modules fall into two broad categories: global and local. Global properties are those which universally bound or otherwise restrict algebraic behaviors present in the constituent modules $M(G)$. Local properties, on the other hand, are those that can be observed when one limits their attention to the modules $M(G)$, where $G$ ranges within certain natural families of graphs. We note that the study of the category $\cGop$ and its representations is still fairly new, and thus the following list should by no means seen as exhaustive. We believe that future study into understanding the inner mechanisms of finitely generated $\cGop$-modules is a very interesting direction for future research.

To begin, we have the following, which is essentially just a reformulation of the definition of finite generation.

\begin{theorem}
Let $M$ be a finitely generated $\cGop_{\leq g}$ Then there exists a non-negative integer $N$ such that for all graphs $G$ one has
\[
M(G) = \colim_{G' < G} M(G'),
\]
where the colimit ranges over all proper minors $G'$ of $G$ that have no more than $N$ edges.
\end{theorem}

One may view this theorem as stating that given a finitely generated $\cGop_{\leq g}$-module $M$, the
presentation of $M(G)$ becomes standardized once $G$ has sufficiently many edges. Our next result relates with
how fast $\cGop_{\leq g}$-modules can grow. For the following statement, we write $e(G)$ for the number of edges of $G$, and $v(G)$ for the number of vertices.

\begin{theorem}\label{bound}
Let $M$ be a finitely generated $\cGop_{\leq g}$-module over a field $k$, and assume that the generators of $M$ have no more than $N$ edges, for some non-negative integers $N$. Then there exists a polynomial $P \in \Q[x,y]$ of degree at most $N$ such that,
\[
\dim_k M(G) \leq \tau(G)P(e(G),v(G)),
\]
where $\tau(G)$ is the number of spanning trees of $G$.
\end{theorem}

\begin{proof}
For simplicity during this proof, write $g(G) = e(G)-v(G)+1$. It will suffice to prove the theorem in the case where $M$ is the principal projective module $P_{G'}$, for some fixed graph $G'$. In this case, for any graph $G$, the dimension of $M(G)$ is precisely the number of minor morphisms from $G$ to $G'$. Any such morphism, up to an automorphism of $G'$, can be determined in the following way: First one picks a spanning tree of $G$ in which all of the contractions will take place. One then chooses $g(G')$ edges not to delete outside of this spanning tree, and $e(G')-g(G')$ edges to not contract within the spanning tree. Therefore,
\[
\dim_k M(G) \leq |\Aut(G')| \tau(G) \binom{g(G)}{g(G')}\binom{v(G)-1}{e(G')-g(G')}.
\]
This then implies,
\[
\dim_k M(G) \leq |\Aut(G')| \tau(G) g(G)^{g(G')}(v(G)-1)^{e(G')-g(G')}
\]
as desired.
\end{proof}

\begin{remark}
The bound of Theorem \ref{bound} is an improvement of a bound found in \cite{PR2}, and can be seen to be sharp. Consider the principal projective module $P_{\bullet},$ over the graph with no edges. In this case, a minor morphism from a graph $G$ to the graph $\bullet$ is precisely determined by a choice of spanning tree for $G$. In particular,
\[ 
\dim_k P_{\bullet}(G) = \tau(G)
\]

The example of the dimension growth of $P_{\bullet}$ is also notable, as it illustrates just how complicated
computing the dimensions in finitely generated $\cGop_{\leq g}$-modules can be. In other words, while one can
in principal compute $\dim_k(P_{\bullet}(G))$ for any graph $G$ using, for instance, the Matrix Tree Theorem,
this is fairly non-trivial counting. Moreover, $P_\bullet$ is the simplest of the principal projective
modules, and this is not even to mention the kinds of behaviors present in the dimensions of the submodules of
principal projective modules. It is for this reason, as we shall see, that it is often times more
mathematically fruitful to consider what we call the local properties of $\cGop_{\leq g}$-modules.
\end{remark}

Our final global property of finitely generated $\cGop_{\leq g}$-modules relates with the kinds of torsion that can appear in the modules $M(G)$.

\begin{theorem}\label{boundedTorsion}
Let $M$ be a finitely generated $\cGop_{\leq g}$-module over $\mathbb{Z}$. Then there exists a non-negative integer $\epsilon$ such that for all graphs $G$ of genus at most $g$, the torsion part of $M(G)$ has exponent at most $\epsilon$.
\end{theorem}

\begin{proof}
	If $ t \in M(G) $ is a torsion element, then for any minor morphism $ \phi: G' \to G $,
	the induced map $ M(G) \to M(G') $ must send $ t $ to a torsion element. Thus, we have a $
	\cGop_{\leq g}$-submodule  $ T $ of $ M  $, where $ T $ sends a graph $ G $ to the torsion part of $
	M(G) $.  By assumption $M(G)$ is finitely generated, so Theorem
	\ref{thm:weakcatgraphminor} tells us that $ T $ is also finitely generated, by say the graphs $ G_1 , \dots , G_{n} $.
	We may take $ \epsilon $ to be the least common multiple of the annihilators of $ T(G_1), \dots, T(G_n) $.
\end{proof}

Moving on, we next consider two local consequences of finite generation. Our first such result describes the growth of the modules associated to families of graphs obtained through sprouting and subdividing.

\begin{theorem}[Corollaries 4.5 and 4.7, \cite{PR2}]\label{basicallyOI}
Let $M$ be a finitely generated $\cGop_{\leq g}$-module over a field $k$, and let $G$ be a fixed graph, with a distinguished collection of vertices $v_1,\ldots,v_r$, and edges $e_1,\ldots,e_s$. We write $G^{(n_1,\ldots,n_r)}$ for the graph obtained from $G$ by attaching $n_i$ leaves to the vertex $v_i$, and $G_{(m_1,\ldots,m_s)}$ for the graph obtained from $G$ by subdividing the edge $e_j$, $m_j$ times. Then there exist polynomials $P_1(n_1,\ldots,n_r)$ and $P_2(m_1,\ldots,m_s)$ such that
\begin{align*}
\dim_k M(G^{(n_1,\ldots,n_r)}) &= P_1(n_1,\ldots,n_r)\\
\dim_k M(G_{(m_1,\ldots,m_s)}) &= P_2(m_1,\ldots,m_s),
\end{align*}
for all vectors $(n_1,\ldots,n_r)$ and $(m_1,\ldots,m_s)$ whose each component is sufficiently large.
\end{theorem}

Through subdivision one can, for instance, study $M(G)$ as $G$ ranges within the cycle graphs. This approach was particularly useful in the case of Kazhdan-Lusztig polynomials of graphical matroids (see \cite{PR2}). Later, we will relate a space considered by Farb Wolfson and Wood to sprouting on the two vertices of a single edge.

Our final local property relates with restriction to the family of trees. Recall that a \textbf{Dyck-path} is a properly nested set of parentheses. Equivalently, it is a word in the symbols $1$ and $-1$ such that all partial sums beginning from the start of the word are never negative. To each Dyck-path $w$, one may associate a (rooted and planar) tree as follows: reading $w$ from left to right, draw a new edge going upward each time a left parenthesis is read, while you backtrack down the nearest edge whenever a right parenthesis is read. For instance, the Dyck-path $(()())$ is associated to the tree that looks like the letter $Y$. Importantly, this association is not one-to-one -- For instance, $()()()$ is also associated to the graph that looks like the letter $Y$ -- though every tree arises in this way\footnote{Note however that the association is one-to-one if we instead consider rooted trees with a cyclic order on the edges moving away from the root for every vertex}. We write $T(w)$ for the tree associated to the Dyck-path $w$. We also write $r(w)$ for the the number of left parenthesis in the Dyck-path $w$ -- i.e. half of its length.

\begin{theorem}[Theorem 1, cite{Ram}]
Let $M$ be a finitely generated $\cGop_{0}$-module over a field $k$. Then the generating function
\[
HD_M(t) := \sum_{w} \dim_k(M(T(w)))t^{r(w)},
\]
where the right hand sum is over all Dyck-paths, is algebraic.
\end{theorem}

If $M$ is the module that assigns the field $k$ to every graph, then the above theorem specializes to say that the generating function for the number of Dyck-paths is algebraic. In other words, the generating function of the Catalan numbers is algebraic. This is a very well known fact about these numbers.

\section{Applications to topological combinatorics}\label{sec:applications}

Throughout this section, let $ g \geq 0 $ be a fixed integer. In this section, we present a number of applications of the representation theory of  $\cGop_{\leq g} $ to topological
combinatorics. 

\subsection{The matching complex}
Now, we introduce matching complexes of graphs and show that for any $ i $, the map assigning to
each graph $ G $ with genus at most $ g $ the $ i $th homology group of its matching complex forms a finitely generated $ \cGop_{\leq
	g}
$-module.

%\subsection{The matching complex}
For a graph $ G $, a \textbf{matching on $ G $} is a subset $ S \subset E(G) $ of non-loop edges such that no
two edges in $ S $ share a vertex. Note that any subset of a matching is also a matching and the empty set, $
\varnothing $, is a matching. Thus, the collection of matchings on $ G $ forms a simplicial complex $ \cM (G)
$ whose vertices are the edges of $ G $, which we call the \textbf{matching complex on $ G $}.

The topology of matching complexes on the complete graphs  $ K_n $ and the complete bipartite graphs $ K_{m,n} $ have been well studied.  Much of what is known about the topology of these complexes is outlined in \parencite{WachsSurvey} and \parencite{JonssonBook}. We outline a few notable results on these complexes below.

It was shown by Bj\"{o}rner, Lov{\'{a}}sz, Vre{\'{c}}ica, and {\v{Z}}ivaljevi{\'{c}} in
\parencite{BjorLovVreZiv} that, for $ n \geq 2 $, $ \cM(K_n) $ is $ \nu_n - 1 $ connected and, for $ 1 \leq m
\leq n $, $ \cM(K_{m,n})$ is $ \nu_{n,m} -1 $ connected where 
\[
	\nu_n = \left\lfloor\frac{n+1}{3}\right\rfloor - 1 \text{ and } \nu_{m,n} = \min \left\{ m,
	\left\lfloor\frac{m+n+1}{3}\right\rfloor \right\}
	-1.
\] 
Later, this was result was strengthened by Shareshian and Wachs, who showed that the $ \nu_n $-skeleton of $
\cM(K_n) $ is shellable \parencite{ShareshWachs} and Ziegler, who showed that the $ \nu_{n,m} $-skeleton of $
\cM(K_{m,n}) $ is shellable.

\begin{remark}
In view of our main Theorem \ref{fghommatch}, the above results may seem a bit concerning. Indeed, finite
generation would be meaningless if it was known that the module was eventually --- i.e. perhaps for all graphs
with sufficiently many edges --- constantly zero! To see that this is not the case, consider the tree $G_n$, which has two vertices of degree $n+1$, connected to each other by a single edge. In this case the matching complex of $G_n$ is seen to be one dimensional. Indeed, it is precisely the disjoint union of the complete bipartite graph $K_{n,n}$ and a single point. This particular topological space only has non-trivial homology in degrees 0 and 1, where it is isomorphic to a free group of ranks 2 and $n^2-2n + 1$, respectively. Similar examples can be constructed to show that, in any homological degree, the homology of the matching complex is not eventually constantly zero.
\end{remark}

The rational homology of these complexes is known due to the work of Bouc \parencite{Bouc} and Friedman and
Hanlon \parencite{FriedHan}. However, much is still unknown regarding the torsion in the integral homology of
these complexes. Shareshian and Wachs show that the homology of $ \cM(K_n) $ exhibits 3-torsion for sufficiently
large $ n $ and the homology of $ \cM(K_{m,n}) $ also contains 3-torsion for certain (but infinitely many)
values of $ m $ and $ n $ \parencite{ShareshWachs}. Later, Jonsson showed that 5-torsion in present in the
homology of $ \cM(K_n) $ for
sufficiently large $ n $ and found that there are elements of order 5, order 7, order 11, and order 13
appearing in the homology of $ \cM(K_n) $ for varying values of $ n $ \parencite{JonssonMoreTor}. For graphs other than $ K_n $ and $ K_{m,n} $, not much is known about the topology of their matching complex. 

There is also a natural generalization of the matching complex. Note, that for a graph $ G $, and any subset $
F \subset E(G) $, we have the induced the subgraph $ G_F $ of $ G$.  Namely, $ G_F $ is the graph with $
V(G_F) = V(G) $ and $ E(G_F) = F $. Thus, $ F $ is a matching on $ G $ if and only if each vertex of $ G_F $
has degree at most 1. More generally, for any integer $ d \geq 1 $ we can consider subsets $ F \subset E(G) $
such that each vertex of $ G_F $ has degree at most $ d $. Call such a subset a \textbf{$d$-matching of $G$}.  Note that the collection of all $ d $-matchings on $ G $ forms a simplicial complex which we will denote $ \cM_d (G) $. In particular, $ M_1(G) = \cM (G) $, the matching complex.

In the case where $ G $ is a forest, Singh showed that $ \cM_d (G) $ is either contractible or homotopy
equivalent to a wedge of spheres \parencite{SinghForests}. For $ d \geq 2 $, Jonsson showed that if $ n $ is
sufficiently large, then there exist certain values of $ d $ depending on $ n $, where the homology of $
\cM_d(K_n) $ contains 3-torsion \parencite{JonssonBoundedDeg}. 

In this paper, we approach the problem of studying the homology of $ \cM_d (G) $ for general graphs $ G $ in a
completely new way, by realizing the map sending a graph $ G $ to the homology of $ \cM_d (G) $ as a finitely
generated $ \cGop_{\leq g} $-module.

To do this, we will need a particular $ \cGop_{\leq g} $-module. Let $ M_{E} $ be the $ \cGop_{\leq g} $-module which assigns to each graph $ G $ the free
$ R $-module with basis indexed by the edges of $ G $. For each minor morphism $ \phi : G \to G' $, the
natural inclusion $ \phi^*: E(G') \to E(G) $ gives us a natural inclusion $ M_{E}(\phi) : M_E(G') \to M_{E}(G)
$. We will call $ M_{E}$, the \textbf{edge module}. This $ \cGop_{\leq g} $-module satisfies the following very important property.

\begin{lemma}\label{fg_edge}
	For any $ i $, the $ i $th tensor power of the edge module, $ M_E^{\otimes i} $, is a finitely generated $
	\cGop_{\leq g}$-module.
\end{lemma}
\begin{proof}
		For any graph $ G $, $ M_{E}^{\otimes i}(G) $ has basis given by $ i $-tuples of edges of $ G $. Let $G$
		be a graph with strictly more than $i$ edges. Then for any tuple $( e_1, \dots, e_{i})$ of edges of $ G $,
		one may find an edge $e$ of $G$ that is not among the $e_j$. By contracting $e$, or deleting it in the
		case where $e$ is a loop, one obtains a minor morphism $\phi: G \rightarrow G'$. It is clear that the
		tuple $( e_1, \dots, e_{i})$ will be in the image of the map induced by
		$\phi$. In other words, $ M_{E}^{\otimes i} $ is generated by graphs with at most $ i $ edges, of which
		there are only finitely many up to isomorphism in $ \cGop_{\leq g} $. 
\end{proof}

\begin{theorem}\label{fghommatch}
	For any $ i \geq 0 $ and any $ d \geq 1 $, $ H_i(\cM_d(-);\Z) $ is a finitely generated $ \cGop_{\leq g} $-module.
\end{theorem}
\begin{proof}
	For this proof, we are working with $ \Z $-coefficients, though we will supress this from the notation. 

	First, we argue that $ H_i(\cM_d(-)) $ is in fact a $ \cGop_{\leq g} $-module. Recall that given a minor morphism $
	\phi : G \to G' $, we have an induced inclusion on edge sets $ \phi^*: E(G') \to E(G) $ sending an edge in $
	G' $ to the unique edge in its preimage under $ \phi $. Now, let $ F' \subset E(G') $.  If some vertex $ v
	\in V(G) $ were incident to more than $ d $ edges in $ \phi^*(F') $, then $ \phi(v) \in V(G') $ would be
	incident to more than $ d $ edges in $ F' $. Thus, if $ F' $ is a $ d $-matching on $ G' $, then $
	\phi^*(F') $ is a $ d $-matching on $ G $. We may therefore, consider $ \phi^* $ as a function 
	\[
		\phi^*: \{\text{$d$-matchings on $G'$}\} \to \{\text{$d$-matchings on $G$}\} .
	\] 
	Furthermore, this map is compatible with taking boundaries of simplices and so we get a chain map 
	\[
		\phi_\bullet : C_\bullet(\cM_d(G')) \to C_\bullet(\cM_d(G))
	\] 
	where $ C_\bullet(\cM_d(G)) $ denotes the simplicial chain complex of $ \cM_d(G) $.  This induces a map $
	H_i(\cM_d(G')) \to H_i(\cM_d(G)) $.  The above construction is compatible with taking compositions of minor
	morphisms and so indeed $ H_i(\cM_d(-)) $ is a $ \cGop_{\leq g} $-module. 

	To see that it is finitely generated, note that following the above construction, for any $ i $, $
	C_i(\cM_d(-)) $ also forms a $ \cGop_{\leq g} $-module and $ H_{i}(M_{d}(-)) $ is a subquotient of $ C_{i}(M_{d}(-))
	$.  Because $ C_i(\cM_d(-)) $ is a submodule of $ \bigwedge^i M_E $, the $ i $th wedge power of the edge
	module, and $ \bigwedge^i M_E $ is a quotient of $ M_E^{\otimes i} $, we see that $ H_i(M_{d}(-)) $ is a
	subquotient of $ M_{E}^{\otimes i} $.  By Lemma \ref{fg_edge}, $ M_{E}^{\otimes i} $ is a finitely generated
	$ \cGop_{\leq g} $-module and so Theorem \ref{thm:weakcatgraphminor} tells us $ H_i(\cM_d(-)) $ is itself finitely generated. 
\end{proof}

This theorem leads to an immediate corollary regarding the torsion that can appear in the homology which says
that for fixed $ i $ and $ d $, there is some uniform maximum torsion that can appear in $ H_i(\cM_d(G)) $ as
$ G $ ranges over all graphs of genus at most $ g $. 

\begin{corollary}\label{torsionmatch}
	For any $ i \geq 0 $ and any $ d \geq 1 $, there exists a positive integer $ \epsilon_{i,d,g} $ such that,
	for any graph $ G $ with genus at most $ g $, the torsion part of $ H_i(\cM_d(-)) $ is annihilated by $ \epsilon_{i,d} $. 
\end{corollary}

\begin{proof}
This follows immediately from the previous theorem as well as Theorem \ref{boundedTorsion}.
\end{proof}

\subsection{Other graph complexes} \label{othergraph}

The matching complex of a graph is just one example of a simplicial complex on the edges of a graph. More
generally, one can consider what is called a \textbf{monotone graph property}. This is a collection $
\mathcal{P} $ of graphs, closed under isomorphisms, such that if $ G \in \mathcal{P}$ and $ G' $ is another
graph with $ V(G') = V(G) $ and $ E(G') \subset E(G) $, then $ G' \in \mathcal{P} $. In other words, $
\mathcal{P} $ is closed under edge deletions. Note that, for any graph $ G $, we obtain a simplicial complex $ \Delta_\mathcal{P}(G) $ on the edges of $ G $ 
where the $ n $-simplices of $ \Delta_\mathcal{P}(G) $ correspond to graphs $ G' \in \mathcal{P}$ with $ V(G') = V(G) $, $ E(G') \subset E(G) $, and $ |E(G') | = n+1 $.  Intuitively if we identify a graph in $ \mathcal{P} $ with its set of edges, the $ n $-simplices are just sets of $ n+1 $ edges of $ G $ in $ \mathcal{P} $.  See \parencite{JonssonBook} for a comprehensive reference on graph complexes. In particular \parencite[Table 7.1]{JonssonBook} gives a list of monotone graph properties and what is known about the homotopy type of their corresponding simplicial complexes. 

Recall that a minor morphism $ \phi: G \to G' $ induces an inclusion $ \phi^*: E(G') \to E(G) $. Suppose that
$ \mathcal{P} $ is a monotone graph property with the extra condition that, if $\phi: G \to G' $ is a minor
morphism and $ H' $ is a simplex in $ \Delta_\mathcal{P}(G') $ (that is, $ H' \in \mathcal{P} $, $ V(H') =
V(G') $, and $ E(H') \subset E(G') $), then the subgraph of $ G $ induced by the image $ \phi^*(E(H')) $ is
also in $ \mathcal{P} $. Call such a $ \mathcal{P} $ a \textbf{$ \cGop_{\leq g} $-monotone graph property}. In rough
terms, a $ \cGop_{\leq g} $-monotone graph property is a monotone graph property that is also preserved under
``uncontracting'' edges.  Note that by construction, if $ \mathcal{P} $ is a $ \cGop_{\leq g} $-monotone graph property, then for any minor morphism $ \phi: G \to G' $, the $ n $-simplices of $ \Delta_\mathcal{P}(G') $ naturally include in the $ n $-simplices of $ \Delta_\mathcal{P}(G)$. This observation and the argument used in the proof of Theorem \ref{fghommatch} yields the following result.

\begin{proposition}
	Let $ \mathcal{P} $ be a $ \cGop_{\leq g} $-monotone graph property. For any $ i $, the assignment of $ G $
	to the $ i $th simplicial homology group $ H_i(\Delta_\mathcal{P}(G)) $ forms a finitely generated $ \cGop
_{\leq g} $-module.
\end{proposition}

Thus, for any $ \cGop_{\leq g} $-monotone graph property, one would obtain an analogous result to Corollary
\ref{torsionmatch} about the torsion that can appear in the $ i $th homology group of $ \Delta_\mathcal{P}(G)
$ as $ G $ ranges over all graphs with genus at most $ g $.

\subsection{Commutative algebra of graph complexes}

Famously, the study of simplicial complexes is, in a formal sense, dual to the commutative algebra of square-free monomial ideals through the Stanley-Reisner correspondence. We consider this perspective in this section.
Fix a field $ K $. Recall that our edge functor $ E: \cGop_{\leq g} \to \FI $ gives rise to another functor $
\cA_E $ from $ \cGop_{\leq g} $ to $ K $-algebras where \[
\cA_{E}(G) = K[ x_{e} \,|\, e \in E(G) ].
\]

In \parencite{NaRo}, Nagel and R\"omer show there are uniform bounds on the degrees for nonvanishing graded
Betti numbers for certain families of ideals that form modules over a particular $ \mathrm{FI} $-algebra.
Using similar techniques, we start by describing some of the homological algebra for $ \cA_E $-modules and
prove analogous results about the graded Betti numbers for families of ideals forming modules over the edge
algebra $ \cA_E $. 

\begin{lemma}\label{proj_resolution}
	For any finitely generated $ \cA_E $-module $ M $ there exists a projective resolution $ F_\bullet $ of $ M $
	by finitely generated $ \cA_E $-modules.
\end{lemma}
\begin{proof}
	By Lemma \ref{presentations}, since $ M $ is finitely generated, there exists graphs $ G_1,\dots,G_n $ and a surjection 
	\[
		\bigoplus_{i = 1}^{n} P^{\cA_E}_{G_i} \to M
	\]
	where $ P_{G_i} $ is the principal projective module at $ G_i $ as defined in Section \ref{sec:modoveralg}.
	Let $ F_0 = \bigoplus_{i = 1}^{n} P^{\cA_E}_{G_i} $. The kernel $ K $ of this surjection must be finitely generated by Theorem \ref{thm:grobpair} and so again, we may find finitely many graphs $ H_1, \dots, H_m $ such that there is a surjection 
	\[
		\bigoplus_{i=1}^{m} P^{\cA_E}_{H_i} \to K.
	\] 
	Let $ F_1 = \bigoplus_{i = 1}^{m} P^{\cA_E}_{H_i} $. Continuing in this fashion, we build the entire projective resolution $ F_\bullet $ of $ M $. 
\end{proof}

Given two finitely generated $ \cA_E $-modules, $ M $ and $ N $, we may take their tensor product over $ \cA_E
$ by taking the pointwise tensor product. This gives an $ \cA_E $-module $ M \otimes_{\cA_E} N $. Explicitly, for any graph $ G $, 
\[
	(M \otimes_{\cA_E} N)(G) = M(G) \otimes_{\cA_E(G)} N(G),
\] 
where the morphisms are defined in the obvious way.  Now, take a projective resolution $ F_\bullet $ of $ M $
as in Lemma \ref{proj_resolution}. Tensoring this resolution with $ N $ gives a chain complex $ F_\bullet
\otimes_{\cA_E} N $ of $ \cA_E $ modules. We define $ \Tor_{i}(M,N) $ to be the $ \cA_E $-module where
\[
	\Tor_{i}(M,N)(G) = H_{i}(F_\bullet(G) \otimes_{\cA_E(G)} N(G)).
\]

For each $ G $, $ \cA_E(G) $ is a graded $ K $-algebra, where $ x_{e} $ is in degree 1 for each $ e \in
E(G) $.  A \textbf{graded $ \cA_E $-module} $ M $ is an $ \cA_E $-module such that each $ M(G) $ is a
graded $ \cA_E(G)
$-module and for any for any minor morphism $ \phi : G \to G' $, the image under $ M(\phi) $ of the degree $ i
$ part of $ M(G') $, denoted $ M (G')_{i} $, is contained in $ M(G)_{i} $.

\begin{remark}\label{graded_proj_resulotion}
	Given a finitely generated graded $ \cA_E $-module $ M $, we may upgrade the surjection of Lemma
	\ref{proj_resolution} to a surjection of graded $ \cA_E $-modules in the following way. Suppose $ M $ is generated by
	the elements $ m_1 , \dots, m_{n} $ where for each $ i $, $ m_{i} \in M(G_i) $. We may assume that each $
	m_i $ is a homogeneous element of $ M(G_i) $ of degree $ d_{i} $. Then, for each $ i $, we give a grading to
	the principal projective $ \cA_E $-module $ P_{G_i}^{\cA_E} $ the grading where $ b_{\phi} \in
	P_{G_i}^{\cA_E}(G) $ is homogeneous of degree $ d_{i} $ for any graph $ G $, and minor morphism $ \phi: G
	\to G_i $. Then, the morphism of $ \cA_E $-modules 
	\[
		\bigoplus_{i = 1}^{n} P_{G_i}^{\cA_E} \to M
	\]
	that sends $ b_{\id_{i}} $ to $ m_{i} $ is a surjection of graded $ \cA_E $-modules. With this fact, we see
	that the projective resolution of Lemma \ref{proj_resolution} may, in fact, be upgraded to a a projective
	resulotion of graded $ \cA_E $-modules. 
\end{remark}

Let $ I_{E} $ be the $ \cA_E $-submodule of $ \cA_E $ itself where $ I_E(G) $ is the ideal $ \langle x_{e} | e
\in E(G) \rangle \subset \cA_E(G) $ for each graph $ G $. Call $ I_{E} $ the \textbf{edge ideal of $\cA_E$}.
Then, $ K_{\cA_E} = \cA_E/I_E $ is the graded $ \cA_E $-module taking each graph to the base field $ K $ in
degree 0, where all morphisms are the identity and for each graph $ G $, monomials of $ \cA_E(G) $ act by
zero, and elements of $ K \subset \cA_E(G) $ act by multiplication. 

\begin{lemma}\label{tor_fg}
	If $ M $ is a finitely generated graded $ \cA_E $-module, then $ \Tor_{i}(M , K_{\cA_E}) $ is also a
	finitely generated graded $ \cA_S $-module for all $ i $. 
\end{lemma}
\begin{proof}
	Let $ F_\bullet $ be a projective resolution of $ M $ where each term is finitely generated as in Lemma
	\ref{proj_resolution}. By Remark \ref{graded_proj_resulotion} we may assume $ F_\bullet $ is a projective
	resolution by graded $ \cA_E $-modules. For each $ i $, $ F_i \otimes K_{\cA_E} $ is isomorphic as an $
	{\cA_E} $-module to the module $ F_i/ I_E F_i $ taking a graph $ G $ to the $ {\cA_E}(G) $-module $ F_i(G) /
	I_E(G) F_i(G) $.  Thus, $\Tor_i(M, K_{\cA_E})$ is a subquotient of a finitely generated $ {\cA_E} $-module and is thus, finitely generated by Theorem \ref{thm:grobpair}.
\end{proof}

Given a finitely generated graded $ {\cA_E} $-module $ M $, a graph $ G $, and integers $ i, a \geq 0 $, the \textbf{$ i\text{th} $ graded Betti number of $M$ in degree $ a $ with respect to $ G $} is defined to be 
\[
	\beta^G_{i,a}(M) := \dim_{K} (\Tor_{i}(M, K_{\cA_E})(G))_a.
\]

\begin{theorem}\label{finite_betti}
	Let $ M $ be a finitely generated graded $ {\cA_E} $-module and fix $ i \geq 0 $. Then, there exists an integer $ n_i $ such that for any graph $ G $, 
	\[
		\beta^G_{i,a}(M) = 0 \quad \text{for all} \quad a > n_i.
	\] 
\end{theorem}
\begin{proof}
	By Lemma \ref{tor_fg}, we know that $ T_i = \Tor_{i}(M , K_{\cA_E}) $ is a finitely generated $ {\cA_E} $ module.  Therefore, we can find a finite list $ t_1 , \dots , t_{n} $ where $ t_j \in M(G_j) $ for some (not necessarily distinct) graphs $ G_1, \dots, G_n $, such that for any graph $ G $, $ T_i(G) $ is generated by the images of the $ t_j $ under maps induced by minor morphisms $ G_j \to G $. Moreover, we may assume that each $ t_j $ is homogeneous of degree $ d_j $. Then, define $ n_i := \max \{d_{j}\}_{j=1}^{n} $. We know that $ \beta^G _{i,a}(M) = \dim_{K} (T_i(G))_a $ is the number of generators of degree $ a $ in a minimal homogeneous generating set of $ T_i(G) $ (see for instance \parencite[Lemma 1.32]{MiStu}). Thus, since the images of the $ t_j $ under maps induced by minor morphisms $ G_j \to G $ contain a minimal homogeneous generating set of $ T_i(G) $, if $ a > n_i $, we must have $ \beta^{G}_{i,a}(M) = 0 $. 
\end{proof}

For any fixed graph $ G $, we can consider $ {\cA_E}(G) = K [x_{e} | e \in E(G) ] $ as an $ \N^{E(G)} $-graded
$ K $-algebra where $ x_{e} $ is in degree $ \mathbf{v}_e \in \N^{E(G)} $ which has a 1 in the $ e $
coordinate and 0's elsewhere. We call this grading the \textbf{edge-grading on $ {\cA_E}(G) $}. If $ M $ is a
graded $ {\cA_E} $-module such that for each graph $ G $, $ M(G) $ is in fact an $ \N^{E(G)} $-graded $
{\cA_E}(G) $-module, we call $ M $ an \textbf{edge-graded $ {\cA_E} $-module}.  
Note that $ K_{\cA_E} $ is an edge-graded $ {\cA_E} $-module where $ K_{\cA_E}(G) $ is concentrated in degree
$ 0 \in \N^{E(G)} $. Thus, if $ M $ is an edge-graded $ {\cA_E} $-module, this induces an edge-graded $
{\cA_E} $-module structure on $ \Tor_{i}(M , K_{\cA_E}) $ for any $ i $.

Now, given an edge-graded $ {\cA_E} $-module $ M $, a graph $ G $, an integer $ i \geq 0 $, and $ \mathbf{a} \in \N^{E(G)} $, we can consider the \textbf{$i^{\text{th}}$ edge-graded Betti number of $M(G)$ in degree $ \a $} denoted $ \beta_{i , \a}(M(G)) $ where 
\[
	\beta_{i , \a}(M(G)) = \dim_{K} (\Tor_{i}(M , K_{\cA_E})(G))_{\a}. 
\]
Let $ \operatorname{sum}( \a ) $ denote the sum of the entries of $ \a $. We see that 
\[
	\beta^G_{i, a}(M) = \sum_{\operatorname{sum}(\a) = a} \beta_{i, \a}(M(G)). 
\]
By Theorem \ref{finite_betti}, we have the following immediate result.

\begin{corollary}\label{finite_edge_betti}
	Let $ M $ be a finitely generated edge-graded $ {\cA_E} $-module and fix $ i \geq 0 $. Then, there exists an integer $ n_i $ such that for any graph $ G $,
	\[
		\beta_{i,\a}(M(G)) = 0 \quad \text{whenever} \quad \operatorname{sum}(\a) > n_i.
	\]
\end{corollary}

Let us now consider the case where $ M $ is a finitely generated edge-graded $ {\cA_E} $-module such that for each
graph $ G $, $ M(G) $ is a square free monomial ideal of $ {\cA_E}(G) = K [x_{e} | e \in E(G)] $. In this case,
using the Stanely-Reisner construction, one can associate to $ M(G) $ a simplicial complex $ \Delta_{M}(G) $
on the set $ E(G) $. Namely, the simplices of $ \Delta_{M}(G) $ are given by squarefree monomials of $ {\cA_E}(G) $
not contained in $ M(G) $. Now, let us identify each subset $ \sigma \subseteq E(G) $ with its indicator
vector in $ \N^{E(G)} $ so that each such subset of $ E(G) $ corresponds to a squarefree degree in the
edge-grading of $ {\cA_E}(G) $.  The commutative algebra of $ M(G) $ is intimately linked with the topology and
combinatorics of $ \Delta_{M}(G) $.  See \parencite{Stanley1983} for an in-depth treatment of this relationship. In particular, there is a nice relationship between the edge-graded Betti numbers of $ M(G) $ and the simplicial complex $ \Delta_{M}(G) $ due to Hochster (see \parencite[Corollary 5.12]{MiStu}). 

\begin{theorem}[Hochster's formula]
	The nonzero Betti numbers of $ M(G) $ lie only in square free degrees, namely degrees corresponding to subsets $ \sigma \subset E(G) $. Furthermore,
	\[
		\beta_{i , \sigma}(M(G)) = \dim_{K} \tilde{H}^{|\sigma| - i - 2}( \Delta_M(G) |_{\sigma} ; K)
	\] 
	where $ \Delta_M(G) |_{\sigma} = \{\tau \in \Delta_M(G) \, | \, \tau \subseteq \sigma\} $. 
\end{theorem}

Hochster's formula together with Corollary \ref{finite_edge_betti} give the following.

\begin{corollary}\label{finite_betti_cohom}
	Fix $ i \geq 0 $. There exists an integer $ n_{i} $ such that for any graph $ G $, if $ \sigma \subseteq E(G) $ with $ | \sigma | > n_{i} $, then 
	\[
		\dim_{K} \tilde{H}^{| \sigma | - i - 2}(\Delta_M(G) |_{\sigma} ; K) = 0.
	\]
\end{corollary}

As an application of the above results, we first define a specific $ {\cA_E} $-module which will be denoted $
I_{\LC} $. For a graph $ G $, the \textit{edge ideal} of $ G $ is the ideal $ I \subseteq K [x_v | v \in V(G)]
$ generated by monomials $ x_vx_w $ where $ v $ and $ w $ are connected by an edge in $ E(G) $. Given a graph
$ G $, define the \textbf{line graph of $ G $} denoted $ \mathcal{L}(G) $ to be the simple graph whose
vertices are the the edges of $ G $ and two vertices of $ \mathcal{L}(G) $ are adjacent if and only if they
share a vertex in $ G $. Then, define the \textbf{complement line graph of $G$}, denoted $ \LC(G) $, to be the
complement of $ \mathcal{L}(G) $. Namely, the vertices of $ \LC(G) $ are the edges of $ G $ and two vertices
of $ \LC(G) $ are adjacent if and only if they do not share a vertex in $ G $.

\begin{example}
If $G$ is the star graph - i.e. the tree with one vertex of degree $n$ and all other vertices of degree 1 -  then $\LC(G)$ is immediately seen to be a disjoint collection of points. On the other hand, if $G$ is the complete graph $K_n$, then $\LC(G)$ is the Kneser graph $K(n,2)$. indeed, certain authors refer to line graph complements as being generalized Kneser graphs for this reason \parencite{DaMeuMi}.
\end{example}

Now, let $ I_{\LC} $ be the $ {\cA_E} $-module taking a graph $ G $ to the edge ideal of $ \LC(G) $.  We see that \[
	I_{\LC}(G) = \langle x_{e} x_{f} | e, f \in E(G)  \text{ don't share a vertex} \rangle \subseteq {\cA_E}(G) .
\] Indeed, $ I_{\LC}(G) $ is a square free monomial ideal of $ {\cA_E}(G) $. Furthermore, if $ \phi : G \to G'
$ is a minor morphism, and $ e, f \in E(G) $ don't share a vertex, then $ \phi^*(e), \phi^*(f) \in E(G ')$
don't share a vertex so $ I_{\LC} $ does, in fact, give an $ {\cA_E} $-module of square free monomial ideals. 

To see that $ I_{\LC} $ is a finitely generated $ {\cA_E} $-module, note that it is a submodule of the edge ideal $
I_{E} $. We see $ I_{E} $ is finitely generated by the graphs $ G_1 $ and $ G_2 $, where $ G_1 $ consists of
two vertices and a single edge connecting them and the graph $ G_2 $ consists of a single vertex and a loop at
that vertex. This is because for any graph $ G $ and any edge $ e \in E(G) $, we can find a minor morphism  $
G \to G_i $ for some $ i \in \{1,2\} $ such that $ e $ gets sent to the single edge in $ G_i $. Thus, the
Noetherian property in Theorem \ref{thm:grobpair} tells us that $ I_{\LC} $ is a finitely generated $ {\cA_E} $-module.

Applying Corollary \ref{finite_edge_betti} for any fixed $ i $, we immediately get the existence of a bound on the degrees of the nonzero edge-graded Betti numbers $ \beta_{i , \a}(I_{\LC}(G)) $ that is uniform as $ G $ ranges over all graphs. 
%Graded Betti numbers of edge ideals of graphs have been studied in \parencite{Kimura2014}, \parencite{HiKiMa}, \parencite{SiVer}, \parencite{RaSi}.  
Moreover, Corollary \ref{finite_betti_cohom} tells us about the cohomology of some subcomplexes of the
simplicial complexes $ \Delta_{I_{\LC}}(G) $ as $ G $ ranges over all graphs. We note that $\Delta_{I_{\LC}}(G)$ is precisely the flag complex - or clique complex - of the line graph $\mathcal{L}(G)$.

\subsection{Linear subspace arrangements of line graph complements}

In this section we study the cohomology of a certain family of hyperplane arrangement complements. In particular, we prove a finite generation result that recovers and expands upon similar results present in \parencite{FWW}.

Let $ G = (V,E) $ be a graph, let $ d $ be a positive integer, and let $ K $ be either $ \C $ or $ \R $. For
each $ e \in E $, if $ v ,w \in V $ are the endpoints of $ e $, let $ W_e \subset (K^d)^V $ be the subspace 
\[
	W_e = \{x \in (K^d)^{V} \, | \, x_v - x_w = 0\}.
\]
The collection of all $ W_e $ as $ e $ ranges over the edges of $ G $ is a subspace arrangement called the
\textbf{graphical arrangement of $ G $} denoted $ \mathscr A(G )$. In the case where $ d=1 $, each $ W_e $ is a
hyperplane of $ K^V $. In general, $ W_e $ is a codimension $ d $ subspace of $ (K^d)^{V} $.  Now, let $
\Conf(G, K^d) = (K^{d})^{V} - \mathscr A(G) $ be the space obtained by removing the subspaces of $ \mathscr A(G) $. In the
case where $ G $ is the complete graph $ K_m $, $ \Conf(G , K^{d}) $ is the usual ordered
configuration space of $ m $ points in $ K^{d} $. 

To each graph $ G $, instead of assigning to it the space $ \Conf (G, K^{d}) $, we will assign to it the space
$ \Conf(\LC(G) , K^{d}) $. We see that if $ \phi : G \to G'$ is a minor morphism, then the inclusion $ \phi^*:
E(G') \to E(G) $ is an inclusion $ V(\LC(G')) \to V(\LC(G)) $. Thus, $ \phi^* $ induces a natural projection 
\[
	\overline{\phi} : (K^{d})^{V (\LC(G))} \to (K^{d})^{V(\LC(G'))} 
\] 
Furthermore, $ \phi^* $ preserves pairs of edges that do not share a vertex. Namely, if $ e_1' ,e_2' \in E(G')
$ do not share a vertex, then $ \phi^*(e_1 ') $ and $ \phi^*(e_2 ') $ do not share a vertex. This observation
yields a natural inclusion $ E(\LC(G')) \to E(\LC(G)) $. Thus, if $ x \in \Conf (\LC(G) , K^d) $, then $
\overline{\phi}(x) \in \Conf(\LC(G), K^{d}) $ so $ \overline{\phi} $ restricts to a map
\[
	\Conf(\LC(G') , K^d) \to \Conf(\LC(G) , K^{d}).
\] 
The construction of the above map is functorial with respect to compositions of minor morphisms. Thus, we have
a functor from $ \cG $ to the category of topological spaces sending a graph $ G $ to $ \Conf(\LC(G), K^d) $.
Then, taking the cohomology ring with $ \Z$ coefficients, gives a functor $ \cGop_{\leq g} $ to the category of abelian groups.

\begin{theorem}
	For any $ i $, the functor $H^{i} (\Conf(\LC(-) , \C^d) ; \Z)$ from $ \cGop_{\leq g} $ to the category of
	abelian groups taking a graph $ G $ to $ H^{i} (\Conf(\LC(G) , \C^d) ; \Z) $ is a finitely generated $ \cGop
_{\leq g}$-module.
\end{theorem}

\begin{proof}
	For a graph $ G $, we know that $ \Conf(\LC(G) , K^d) $ is the complement of a
	subspace arrangement where each subspace has real codimension $ r $ where $ r=d $ when $ K=\R $ and $ r=2d-1 $
	when $ K = \C $. In \parencite[Corollary 5.6]{DeLoungSchultz}, de Loungueville and Schultz give a presentation for the
	cohomology ring of the complement of a subspace arrangement over $ \R $ where each subspace has the same
	codimension. This result tells says that the cohomology ring $ H^*(\Conf(\LC(G) , K^d) ; \Z) $ has the
	following presentation: 
	\[
		H^{*}(\Conf(\LC(G) , K^{d}) ; \Z) \cong \wedge^* \Z^{E(\LC(G))} / I 
	\] when $ r $ is even and 
	\[
		H^{*}(\Conf(\LC(G) , K^{d}) ; \Z) \cong \operatorname{Sym}^* \Z^{E(\LC(G))} / I 
	\] when $ r $ is odd. Moreover, if $ e \in E(\LC(G)) $, then under this isomorphism $ e \in
	H^{r-1}(\Conf(\LC(G) , K^{d})) $. If $ r $ is even, $ I $ is generated by 
	\[
		\sum_{i = 0}^{k} (-1)^{i} \epsilon(e_{j_1} , \dots, \hat{e}_{j_{i}} , \dots , e_{k_{i}}) e_{j_0} \wedge
		\dots \wedge \widehat{e}_{j_i} \wedge \dots \wedge e_{j_{k}} 
	\]
	and if $ r $ is odd, $ I $ is generated by 
	\[
		e^2 \quad \text{and} \quad \sum_{i = 0}^{k} (-1)^{i} \epsilon(e_{j_1} , \dots, \hat{e}_{j_{i}} , \dots
		, e_{k_{i}}) e_{j_0} \dots \widehat{e}_{j_i}\dots e_{j_{k}} 
	\]
	for all $ e \in E(\LC(G)) $ and all sets $ \{ e_{a_{0}} , \dots , e_{a_{k}}\} \subset E(\LC(G)) $ that form
	a cycle. Here, $\epsilon(e_{j_1} , \dots, \hat{e}_{j_{i}} , \dots , e_{k_{i}})$ is a sign coming from a choice of
	orientation on the ambient real vector space. 

	We know cycles of $ \LC(G) $ correspond to sets of edges $ \{f_1 , \dots, f_{k}\} \subset E(G) $ such that $
	f_i$ and $ f_{i+1} $ do not share a vertex and $ f_1 $ and $ f_k $ do not share a vertex. Given a minor
	morphism $\phi: G \to G' $, the inclusion $ \phi^* : E(G') \to E(G) $ preserves pairs of edges that don't
	share a vertex, and so the induced map $ E(\LC(G')) \to E(\LC(G)) $ sends cycles to cycles. Furthermore, the
	inclusion $ \phi^* : E(G') \to E(G) $  gives an inclusion of the ambient real vector space for $ \Conf(E(G'),
	K^d) $ into the ambient real vector space for $ \Conf(E(G) , K^{d}) $ which in particular, allows us to pick
	compatible orientations on these ambient vector spaces.  Thus, the presentation for $ H^*(\Conf(\LC(G)) ,
	\C^d);\Z) $ is compatible minor morphisms. Thus, we see that for each $ i $, $ H^{ir}(\Conf(\LC(-) , \C^d) ;
	\Z) $ as $ \cGop_{\leq g} $-module, is a quotient of the $ i $th tensor power of $ \Z^{E(\LC(G))} $,
	\[
		T^i \Z^{E(\LC(-))}.
	\]
	Finally, we notice that $ \Z^{E(\LC(-))} $ is the submodule of the second tensor power of the edge module $
	M_{E}^{\otimes 2} $, where $ \Z^{E(\LC(G))} $ generated by elements of the form $ e_{i} \otimes e_{j} $ where
	$ e_{i} , e_{j} \in E(G) $ do not share a vertex. Thus, by Theorem \ref{thm:weakcatgraphminor}, we have the desired result.
\end{proof}

\begin{remark}
	If $ d=1 $, $ \Conf(\LC(G), K) $ is the complement of a hyperplane arrangement in $ K^V $.
	When $ K = \C $, $ H^{*}(\Conf(\LC(G) , K)) $ is just the Orlik-Solomon algebra of the
	complex hyperplane arrangement where the generators live in degree 1. If $ d > 1 $, the above presentation shows
	that $ H^{*}(\Conf(\LC(G) , K^{d})) $ is still isomorphic to the aforementioned Orlik-Solomon algebra, only
	now the generators live in degree $ 2d - 1 $.
	When $ K=\R $, $ H^*(\Conf(\LC(G) , K)) $ is the Cordovil algebra of the real hyperplane arrangement. If $ d $ is
	odd, then $ H^*(\Conf(\LC(G) , K^{d})) $ is isomorphic to $ H^*(\Conf(\LC(G) , K)) $, only now the
	generators live in degree $ d-1 $.
\end{remark}

\begin{example}
Let $a,b \geq 1$ and let $G$ be the graph with two vertices of degrees $a+1$ and $b+1$, respectively, connected to one another by a single edge. Put another way, $G$ is two copies of star graphs, of degrees $a$ and $b$, glued together by an edge. Then $\LC(G)$ is easily seen to be a complete bipartite graph $K_{a,b}$, disjoint union a point. Then we have
\[
\Conf(\LC(G) , \C^{d}) = \widetilde{\mathcal{Z}}_{a+b}^D(\C^d) \times \C^d
\]
where $\widetilde{\mathcal{Z}}_{a+b}^D(\C^d)$ are the colored configuration spaces considered by Farb Wolfson and Wood in \parencite{FWW}, with $D$ being the vertices of the complete bipartite graph colored in the obvious way. Observe moreover that the graph $G$ can be seen as an edge with $a$ and $b$ leaves sprouted on its two vertices, respectively. Therefore, Theorem \ref{basicallyOI} implies that our result can be seen as a generalization of the stabilization phenomena observed by \parencite{FWW}.
\end{example}

\printbibliography

@book{JonssonBook,
	AUTHOR = {Jonsson, Jakob},
     TITLE = {Simplicial complexes of graphs},
    SERIES = {Lecture Notes in Mathematics},
    VOLUME = {1928},
 PUBLISHER = {Springer-Verlag, Berlin},
      YEAR = {2008},
     PAGES = {xiv+378},
      ISBN = {978-3-540-75858-7},
}

@article{JonssonBoundedDeg, 
	title={3-torsion in the Homology of Complexes of Graphs of Bounded Degree}, 
	volume={65}, 
	number={4}, 
	journal={Canadian Journal of Mathematics}, 
	publisher={Cambridge University Press}, 
	author={Jonsson, Jakob}, 
	year={2013}, 
	pages={843–862},
}

@misc{MiProRa,
	Author = {Dane Miyata and Nicholas Proudfoot and Eric Ramos},
	Title = {The categorical graph minor theorem},
	Year = {2020},
	Eprint = {arXiv:2004.05544},
}

@article{SaSno,
	AUTHOR = {Sam, Steven V. and Snowden, Andrew},
     TITLE = {Gr\"{o}bner methods for representations of combinatorial categories},
   JOURNAL = {J. Amer. Math. Soc.},
  FJOURNAL = {Journal of the American Mathematical Society},
    VOLUME = {30},
      YEAR = {2017},
    NUMBER = {1},
     PAGES = {159--203},
}

@article{PR,
  title={Functorial invariants of trees and their cones},
  author={Proudfoot, Nicholas and Ramos, Eric},
  journal={Selecta Mathematica},
  volume={25},
  number={4},
  pages={62},
  year={2019},
  publisher={Springer},
}

@article{PR2,
  title={The contraction category of graphs},
  author={Proudfoot, Nicholas and Ramos, Eric},
  Eprint={arXiv:1907.11234},
  year={2019},
}

@article{Bouc,
	year = {1992},
	month = aug,
	publisher = {Elsevier {BV}},
	volume = {150},
	number = {1},
	pages = {158--186},
	author = {S. Bouc},
	title = {Homologie de certains ensembles de 2-sous-groupes des groupes sym{\'{e}}triques},
	journal = {Journal of Algebra},
}

@article{ShareshWachs,
	title={Torsion in the matching complex and chessboard complex},
	volume={212},
	number={2},
	journal={Advances in Mathematics},
	publisher={Elsevier BV},
	author={Shareshian, John and Wachs, Michelle L.},
	year={2007},
	pages={525–570},
}

@article{BjorLovVreZiv,
	year = {1994},
	month = feb,
	publisher = {Wiley},
	volume = {49},
	number = {1},
	pages = {25--39},
	author = {A. Bj\"{o}rner and L. Lov{\'{a}}sz and S. T. Vre{\'{c}}ica and R. T. {\v{Z}}ivaljevi{\'{c}}},
	title = {Chessboard Complexes and Matching Complexes},
	journal = {Journal of the London Mathematical Society},
}

@article{WachsSurvey,
	year = {2003},
	month = oct,
	publisher = {Springer Science and Business Media {LLC}},
	volume = {49},
	number = {4},
	pages = {345--385},
	author = {Michelle L. Wachs},
	title = {Topology of matching,  chessboard,  and general bounded degree graph complexes},
	journal = {Algebra Universalis},
}

@article{FriedHan,
	year = {1998},
	publisher = {Springer Science and Business Media {LLC}},
	volume = {8},
	number = {2},
	pages = {193--203},
	author = {Joel Friedman and Phil Hanlon},
	title = {On the Betti Numbers of Chessboard Complexes},
	journal = {Journal of Algebraic Combinatorics},
}

@article{Jonsson3TorChess,
	year = {2010},
	month = dec,
	publisher = {Springer Science and Business Media {LLC}},
	volume = {14},
	number = {4},
	pages = {487--505},
	author = {Jakob Jonsson},
	title = {On the 3-Torsion Part of the Homology of the Chessboard Complex},
	journal = {Annals of Combinatorics}
}

@article{JonssonMoreTor,
	year = {2010},
	month = jan,
	publisher = {Informa {UK} Limited},
	volume = {19},
	number = {3},
	pages = {363--383},
	author = {Jakob Jonsson},
	title = {More Torsion in the Homology of the Matching Complex},
	journal = {Experimental Mathematics}
}

@article{SinghForests,
	year = {2020},
	month = oct,
	publisher = {Elsevier {BV}},
	volume = {343},
	number = {10},
	pages = {112009},
	author = {Anurag Singh},
	title = {Bounded degree complexes of forests},
	journal = {Discrete Mathematics}
}

@article{DeLoungSchultz,
  year = {2001},
  month = apr,
  publisher = {Springer Science and Business Media {LLC}},
  volume = {319},
  number = {4},
  pages = {625--646},
  author = {Mark de Longueville and Carsten A. Schultz},
  title = {The cohomology rings of complements of subspace arrangements},
  journal = {Mathematische Annalen}
}

@article{NaRo,
  year = {2019},
  month = oct,
  publisher = {Elsevier {BV}},
  volume = {535},
  pages = {286--322},
  author = {Uwe Nagel and Tim R\"{o}mer},
  title = {{FI}- and {OI}-modules with varying coefficients},
  journal = {Journal of Algebra}
}

@Book{MiStu,
 author = {Miller, Ezra and Sturmfels, Bernd},
 title = {Combinatorial commutative algebra},
 publisher = {Springer},
 year = {2005},
 address = {New York},
 isbn = {978-0-387-22356-8},
}

@book{Stanley1983,
	AUTHOR = {Stanley, Richard P.},
     TITLE = {Combinatorics and commutative algebra},
    SERIES = {Progress in Mathematics},
    VOLUME = {41},
   EDITION = {Second},
 PUBLISHER = {Birkh\"{a}user Boston, Inc., Boston, MA},
      YEAR = {1996},
     PAGES = {x+164},
      ISBN = {0-8176-3836-9},
}

@misc{DaMeuMi,
	title={Colorings of complements of line graphs}, 
	author={Hamid Reza Daneshpajouh and Frédéric Meunier and Guilhem Mizrahi},
	year={2020},
	eprint={2003.08255},
	archivePrefix={arXiv},
	primaryClass={math.CO},
}

@article{FWW,
  title={Coincidences between homological densities, predicted by arithmetic},
  author={Farb, Benson and Wolfson, Jesse and Wood, Melanie Matchett},
  journal={Advances in Mathematics},
  volume={352},
  pages={670--716},
  year={2019},
  publisher={Elsevier}
}

@article {RSXX,
    AUTHOR = {Robertson, Neil and Seymour, P. D.},
     TITLE = {Graph minors. {XX}. {W}agner's conjecture},
   JOURNAL = {J. Combin. Theory Ser. B},
    VOLUME = {92},
      YEAR = {2004},
    NUMBER = {2},
     PAGES = {325--357},
}

@article {RSXXIII,
    AUTHOR = {Robertson, Neil and Seymour, Paul},
     TITLE = {Graph minors {XXIII}. {N}ash-{W}illiams' immersion conjecture},
   JOURNAL = {J. Combin. Theory Ser. B},
    VOLUME = {100},
      YEAR = {2010},
    NUMBER = {2},
     PAGES = {181--205},
}

@article{Ram,
  title={Hilbert series in the category of trees with contractions},
  author={Ramos, Eric},
  journal={arXiv preprint arXiv:2007.05669},
  year={2020},
}

\end{document}